\documentclass[amssymb,amsfonts,reqno]{amsart}

\usepackage{amssymb} 
\usepackage{amsmath} 
\usepackage{amsthm} 
\usepackage{mathtools} 
\usepackage{mathrsfs} 
\usepackage{comment} 
\usepackage{bm} 
\usepackage{bbm}
\usepackage{hyperref} 

\def\be#1{\begin{equation}\label{#1}}
\def\bas{\begin{align*}}
\def\eas{\end{align*}}
\def\bi{\begin{itemize}}
\def\ei{\end{itemize}}

\DeclareFontFamily{OT1}{pzc}{}
\DeclareFontShape{OT1}{pzc}{m}{it}{<-> s * [1.10] pzcmi7t}{}
\DeclareMathAlphabet{\mathpzc}{OT1}{pzc}{m}{it}

\theoremstyle{plain}
   \newtheorem{theorem}[subsection]{Theorem}
   \newtheorem*{theorem*}{Theorem}

   \newtheorem{definition}[subsection]{Definition}   
   \newtheorem{remark}[subsection]{Remark}

   \newtheorem{proposition}[subsection]{Proposition}
   \newtheorem*{proposition*}{Proposition}
   
   \newtheorem{lemma}[subsection]{Lemma}

\newcommand{\RR}{\mathbb{R}}
\newcommand{\CC}{\mathbb{C}}
\newcommand{\ZZ}{\mathbb{Z}}

\renewcommand{\vec}[1]{\underline{#1}}

\DeclareMathOperator{\Span}{span}

\author{Reuben Wheeler}
\address{Maxwell Institute of Mathematical Sciences and the School of Mathematics, University of Edinburgh, JCMB, The King's Buildings, Peter Guthrie Tait Road, Edinburgh, EH9 3FD, Scotland}
\email{reuben.wheeler@ed.ac.uk}
\title{Root structures of polynomials with sparse exponents}
\begin{document}
\maketitle
\begin{abstract}For real polynomials with (sparse) exponents in some fixed set, \[ \Psi(t)=x+y_1t^{k_1}+\ldots +y_L t^{k_L}, \] we analyse the types of root structures that might occur as the coefficients vary. We first establish a stratification of roots into tiers, each containing roots of comparable sizes. We then show that there exists a suitable small parameter $\epsilon>0$ such that, for any root $w\in \mathbb{C}$, $B(w,\epsilon|w|)$ contains \emph{at most} $L$ roots, counted with multiplicity. Our analysis suggests the consideration of a rough factorisation of the original polynomial and we establish the closeness of the corresponding root structures: there exists a covering of the roots by balls wherein a) each ball contains the same number of roots of the original polynomial and of its rough factorisation and b) the balls are strongly separated.\end{abstract}

\section{Introduction}
Methods for finding solutions to polynomial equations have a storied history in mathematics, leading to the extension of number systems, the development of algebra, and being of particular importance in the development of numerical analysis \cite{goldstineHist}. Before taking stock of more contemporary developments, let us recall some points in that history most relevant to our analysis. On the theoretical side, we have the fundamental theorem of algebra, which tells us that every degree $n$ complex polynomial has $n$ roots. Nevertheless, Abel and Ruffini established that there exist polynomials of degree $5$ whose solution could not directly expressed in terms of radicals. The work of Galois later established under what conditions all the solutions to polynomial equations could be finitely expressed in terms of radicals \cite{rigGalois}. Despite this, numerical tools still allow us to approximate the solutions to polynomial equations to arbitrary degrees of accuracy.

In harmonic analysis, it is useful even to have rough picture of root structure. We may be interested in the structure of the roots of real polynomials or, indeed, the roots of polynomials over non-Archimedean fields (in a suitable field extension). Such characterisations of root structure can be used to estimate the associated sublevel sets or to bound corresponding oscillatory integrals via the bounds of Phong and Stein \cite{stePho}. 

The analysis contained herein is inspired by work of Kowalski and Wright, \cite{kowWri}, and an unpublished oscillatory integral estimate of Hickman and Wright \cite{wrightHickCom}, Theorem \ref{thm:oscIntEst}. It takes the perspective of root clusters, in the spirit of the famous oscillatory integral estimates of Phong and Stein. We find this to be a timely moment to make these results, which appeared as part of the author's thesis \cite{mythesis}, public, as the authors Hickman and Wright have recently shared work utilising a similar approach \cite{hickWriLP}. Their presentation is focused on the non-Archimedean context, but the broader applicability of their arguments is indicated. Though our analysis is carried out over $\RR$, it is likely that the proofs contained herein are valid over other fields. Indeed, the proofs are likely cleaner in the non-Archimedean setting as is the case in \cite{kowWri} and the recent \cite{hickWriLP}.

 Before preparing to state our main results, let us note of the important connections that have been found with the polynomial root finding literature. Perhaps the most significant are to be found in works by Sch\"onhage \cite{schonhage} and Bini \cite{bini96}, which we discuss momentarily. Also of note is the fact that root finding algorithms must carefully account for the ways in which roots can come close or overlap, which, for example, can cause a break-down in a first order Newton-Raphson iteration. Much of this literature is concerned with the real or complex case and makes use of tools fitted to this context such as Rouch\'e's theorem. Our work makes use of basic metric properties of $\CC$ and its valuation and is thus likely applicable in other fields. Not all complex root-finding techniques make use of tools specialised to that context: there is, for example, an early work by van Vleck \cite{vanVleck} which considers polynomials with sparse exponents and works via the consideration of the vanishing symmetric functions of roots, as we do. Readers interested in a more comprehensive perspective on related ideas in polynomial root finding are directed to, for example, \cite{panHistory} and \cite{mardenGeom}. Following this history, there are now highly efficient procedures for complex polynomial root finding \cite{panNearOptimal}.
 
The approximate factorisation method of Sch\"onhage \cite{schonhage} is an effectively parallelisable algorithm, which can be useful for a localised root finding. Nevertheless, it involves the calculation of complex residues via a costly integration procedure. Our later rough factorisation theorem, which is directly expressed in terms of the polynomial coefficients, may provide a useful initialisation for this method (or, indeed, other root finding algorithms). 

The root finding approach of Bini \cite{bini96} makes use of suitably separated root collections determined via the use of height estimates. Our more explicit tools for determining the root structure of for polynomials with sparse exponents, which appear later in our analysis, makes use of these same height estimates. These height estimates can help us to determine in which annular regions our different collections of roots lie. 
 
In the case of sparse polynomials, our results may provide some useful information for root finding algorithms narrowing the scope of the pathological cases wherein roots coalesce. Indeed, algorithms to identify appropriate root clusters are the subject of contemporary research \cite{imbachPan} \cite{beckerEtAl}. On the one hand, the results of this paper bound the size of root clusters that might need to be considered. On the other hand, the root cell covering appearing arising in our rough factorisation analysis may be viewed as the construction of appropriate root clusters.
 
To prepare the statement of our main theorem, let us now consider a couple of representative examples. Consider the polynomial \[\Psi_1(t)=x+y_1t+y_2t^2+\ldots y_Lt^L,\]
with $x,y_L\neq 0$. We know, by the fundamental theorem of algebra, that $\Psi_1$ has $L$ roots, counted with multiplicity. In particular, for some small $\epsilon>0$, we know that, for any root $w$, at most $L$ roots are contained in $B(w,\epsilon |w|)$, which is, of course, trivial. Now consider the polynomial
\[\Psi_k(t)=\Psi_1(t^k).\]
Corresponding to each root of $w$ of $\Psi_1$, we see that there are $k$ roots of $\Psi_k$, these are the $k$th roots of $w$. As before, we also have that, for some small $\epsilon>0$, and for any root $w$, at most $L$ roots are contained in $B(w,\epsilon |w|)$. See Figures \ref{fig:psi1Roots} and \ref{fig:psikRoots} for a sketch of the roots of $\Psi_1$ and $\Psi_k$, respectively, in the case that $L=4$, $k=5$, and \[\Psi_1(t)=y_L(t-h_1)(t-h_2)(t-h_3)(t+h_4).\]
More specifically, these figures are sketches of root structure in the case where $h_1\gg h_2\approx h_3\gg h_4>0$. The dotted black circle corresponds to roots with modulus $h_1$. The solid blue lines correspond to roots with modulus close to $h_2$. The dashed red line corresponds with roots of modulus $h_4$. However,  for any choice $h_j$ one can see we have the following. For some suitable small $\epsilon>0$ and for any root $w'$ of $\Psi_k$, there are at most $L=4$ roots contained in $B(w',\epsilon |w'|)$.\footnote{In fact, for the sketched case, due to the indicated strong separation of root heights $h_1\gg h_2$ and $h_3\gg h_4$, we have at most $L=2$ roots contained in such a ball. This is suggestive of refinements that can be made to our structural statements under additional restrictions on the polynomial coefficients.} This particular example reflects polynomial root structure more generally, as we outline in Theorem \ref{thm:tierStruc}. 
\begin{figure}[!htb]
   \begin{minipage}{0.49\textwidth}
     \centering
     \includegraphics[width=.95\linewidth]{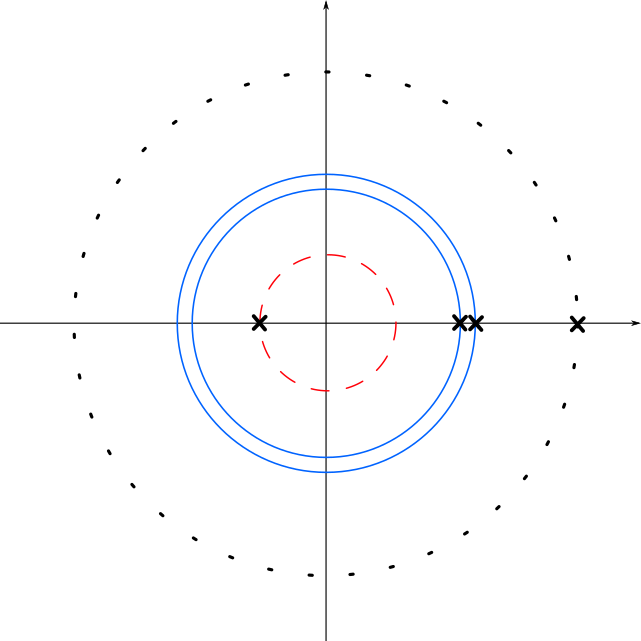}
     \caption{The roots of $\Psi_1$.}
     \label{fig:psi1Roots}
   \end{minipage}\hfill
   \begin{minipage}{0.49\textwidth}
     \centering
     \includegraphics[width=.95\linewidth]{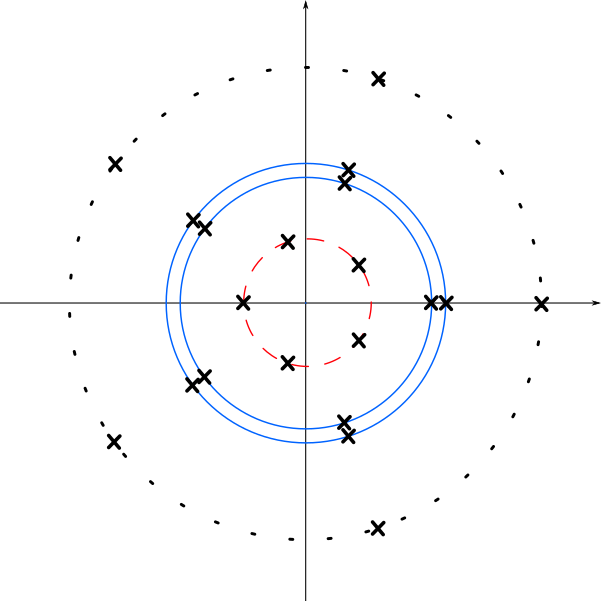}
     \caption{The roots of $\Psi_k$.}
     \label{fig:psikRoots}
   \end{minipage}
\end{figure}

\begin{theorem}\label{thm:tierStruc}
We fix a set of exponents $0=k_0< k_1<k_2<\ldots<k_L$ and consider real polynomials whose exponents are drawn from this set. For any real polynomial $\Psi(t)=x+y_1t^{k_1}+y_2t^{k_2}+\ldots+y_Lt^{k_L}$ with $x,y_L\neq 0$, we have the following structure.
 
The roots of $\Psi$ are stratified into $s$ tiers of roots $\mathcal{T}_1,\mathcal{T}_2,\ldots ,\mathcal{T}_s$, for some suitable $1\leq s \leq L$. The tiers are separated in the sense that, if we take $w_{i}\in \mathcal{T}_i$, then 
\[|w_{1}|\gg |w_{2}|\gg\ldots \gg |w_{s}|.\] 
 
For a suitable $1\leq L(\mathcal{T}_r)\leq L$, at most $L(\mathcal{T}_r)$ roots can cluster about a root: there is some suitable small parameter $\epsilon$ such that, for all $w\in\mathcal{T}_{r}$, there are at most $L(\mathcal{T}_r)$ roots of $\Psi$ in $B(w,\epsilon|w|)$, counted with multiplicity.
\end{theorem}

\begin{remark}Our main structural results are sharp in the sense that, we can construct a polynomial with sparse exponents for which the tiers have repeated roots of the largest degree permitted by our theorems. In particular, the polynomial
\[\Psi(t)=\prod_{r=1}^s(t^k-h_r^{k})^{l_r},\]
with $h_1\gg h_2\ldots \gg h_s$, is such that $L(\mathcal{T}_r)=l_r$ (repeated) roots are contained in $B(w,\epsilon|w|)$ for any choice of $w\in\mathcal{T}_r$, with $\mathcal{T}_r$ consisting of the $k$th roots of $h_r^k$ of which appears with multiplicity $l_r$.
\end{remark}

Let us recall the following structural result of Kowalski and Wright, Theorem 1.6 of \cite{kowWri}. The authors found application of this theorem to bound oscillatory integrals as well as for estimating the measure of sublevel sets of polynomials.
\begin{theorem}\label{thm:kowWri}
Let \[\Psi(t)=a_{k_L}t^{k_L}+a_{k_L-1}t^{k_L-1}+\ldots+a_0=a_d\prod_j(t-z_j)\] be a complex polynomial with $\max_l|a_l|=1$. Suppose that the coefficients satisfy $0<\gamma \leq |a_{k_L-k}|$ and $|a_{k_L-j}|\leq \delta_j(\gamma)$, $0\leq j\leq k-1$ for some $0\leq k \leq k_L$, where $\delta_j$ is some suitably small constant for each $0\leq j\leq k-1$. Then there are exactly $k$ large roots $z_1,z_2,\ldots,z_k$ and the remaining roots are bounded. In particular, with ordered roots $|z_1|\geq|z_2|\geq\ldots\geq|z_{k_L}|$, we have that \[\left(\frac{\gamma}{\max_j\delta_j}\right)^\frac{1}{k_L}\lesssim |z_1|,\ldots,|z_k|\text{ and } |z_{k+1}|,\ldots,|z_{k_L}|\lesssim 1.\]
\end{theorem}
 This result is notably similar to that of van Vleck \cite{vanVleck} (following Montel \cite{montel}), which provides an explicit upper bound for the smallest $k_L-k$ roots, if we follow the above statement.

Further to Theorem \ref{thm:tierStruc}, and analogous to Theorem \ref{thm:kowWri}, we are able to obtain the following refined structural result. 
\begin{theorem}\label{thm:tierStrucFine}We fix a set of exponents $0=k_0< k_1<k_2<\ldots<k_L$ and consider certain polynomials whose exponents are drawn from this set. We consider real polynomials \[\Psi(t)=x+y_1t^{k_1}+y_2t^{k_2}+\ldots+y_Lt^{k_L}\] with $x,y_L\neq 0$ and $\max_{1\leq j \leq L}|y_j|=1$. Let $\gamma\in (0,1]$. We suppose, additionally, that there exists $m$ such that
\begin{equation}\label{eq:refinedCoefControl}\begin{aligned}&|y_m|^\frac{1}{k_m}\geq\gamma ,\\
&\text{and, for }n>m,\; |y_n|^\frac{1}{k_n}\leq \delta,
\end{aligned}\end{equation} where $\delta=\delta(\gamma)>0$ is some suitably small constant. We have the following refined structure for the roots of $\Psi$.

The roots of $\Psi$ may be stratified into $s=s(1)+s(2)$ tiers of roots $\mathcal{T}_1,\mathcal{T}_2,\ldots ,\mathcal{T}_s$, where $1\leq s \leq L$, which are ordered and separated in the sense that, if we take $w_{i}\in \mathcal{T}_i$, then \[|w_{1}|\gg |w_{2}|\gg\ldots \gg |w_{s}|.\]
Furthermore, the tiers have the following additional structure.

 We refer to each tier $\mathcal{T}_r$ with $1\leq r\leq s(1)$ as a large tier. Likewise, we refer to any $\mathcal{T}_r$ with $s(1)< r \leq s(1)+s(2)$ as a small tier. We have that
 \[1\lesssim_{\gamma} |w|\text{ for }w\in \mathcal{T}_r\text{ with }1\leq r\leq s(1)\]
 \[\text{ and } |w|\lesssim_{\gamma} 1\text{ for }w\in \mathcal{T}_r\text{ with }s(1)< r \leq s(1)+s(2).\]
 
At most $L-m+1$ large roots can cluster about a point: there is some suitable small parameter $\epsilon$ such that, for all $w\in\mathcal{T}_{r}$ with $1\leq r\leq s(1)$, there are at most $L-m+1$ roots of $\Psi$ in $B(w,\epsilon|w|)$. Furthermore, in the case that $s(2)\geq 1$, we have the following improvement. For any small root, $w\in \mathcal{T}_r$ with $s(1)< r \leq s(1)+s(2)$,  $B(w,\epsilon |w|)$ can contain at most $m$ roots. For any root $w$ in a large tier, $w\in \mathcal{T}_r$ with $1\leq  r \leq s(1)$, $B(w,\epsilon |w|)$ can contain at most $L-m$ roots.
 \end{theorem}
 
 \begin{remark}
Other explicit refinements of the structure theorem, which are amenable to direct calculation, are possible. These refinements can be achieved according with our tier height estimation procedure, Lemma \ref{lem:heightEstProcedure}. 
\end{remark}

As mentioned above, we use these structural results to obtain bounds on oscillatory integrals with polynomial phases. We consider polynomial phases \begin{equation}\label{eq:polyPhaseDef}\begin{split}\Phi(t)=xt+\frac{y_1}{k_1+1}t^{k_1+1}+\frac{y_2}{k_2+1} t^{k_2+1}+\ldots+\frac{y_L}{k_L+1}t^{k_L+1},\\\text{ with } 1<k_1<k_2<\ldots<k_L
,\end{split}\end{equation}
with $y_L\neq 0$.
Given our main structure result, Theorem \ref{thm:tierStrucFine}, one can make use of an oscillatory integral bound for oscillatory integrals due to Phong and Stein, Theorem \ref{thm:stePho}.  From this, it is possible to obtain an alternative proof of the following unpublished oscillatory integral estimate due to Hickman and Wright, \cite{wrightHickCom}; see \cite{mythesis} for a presentation of their original proof. \begin{theorem}\label{thm:oscIntEst}For oscillatory integrals
\[I(x,y)=\int_{\RR} e^{i\Phi(t)}dt\] with phases $\Phi$ given by \eqref{eq:polyPhaseDef}, we have that \[\left|I(x,y)\right|\lesssim\min_{j=1,2,\ldots,L}|y_j|^{-\frac{1}{k_j+1}},\]
provided $k_1\geq L$.
\end{theorem}

A cluster, $\mathcal{C}$, is a non-empty subcollection of roots of $\Phi'$. We make use of the following result of Phong and Stein, \cite{stePho}.\begin{theorem}\label{thm:stePho}
Suppose that $\Phi$, \eqref{eq:polyPhaseDef}, is a real polynomial such that $y_L\neq 0$ and $\Phi'$ has roots $z_1,z_2,\ldots,z_k$, counted with multiplicity. Then we have the oscillatory integral estimate \[\left|I(x,y)\right|=\left|\int_{\RR} e^{i\Phi(t)}dt\right|\leq C_{k} \max_j\min_{\mathcal{C}\ni z_j}\frac{1}{\left(|y_L|\prod_{l\notin \mathcal{C}}|z_j-z_l|\right)^\frac{1}{|\mathcal{C}|+1}},\]
where the maximum is taken over roots $z_j$ and the minimum over clusters of roots, $\mathcal{C}\subset \left\lbrace z_1,z_2,\ldots,z_k\right\rbrace$, such that  $z_j\in\mathcal{C}$.
\end{theorem}

In fact, our refined structural statement, Theorem \ref{thm:tierStrucFine} can be applied to obtain a refined oscillatory integral estimate. Hickman and Wright established the previous estimate, Theorem \ref{thm:oscIntEst} and, as shown in \cite{mythesis}, their method of proof is capable of establishing the below result.

\begin{theorem}\label{thm:oscIntEstFine} For oscillatory integrals\[I(x,y)=\int_{\RR} e^{i\Phi(t)}dt\] with phases $\Phi$ given by \eqref{eq:polyPhaseDef}, we have that \begin{equation}\label{eq:oscIntEstI}\left|I(x,y)\right|\lesssim\min_{j=1,2,\ldots,L}|y_j|^{-\frac{1}{k_j+1}},\end{equation}
provided $k_1\geq L$.
More generally, if $\max_j|y_j|^{\frac{1}{k_j}}=|y_m|^\frac{1}{k_m}$, then \[\left|I(x,y)\right|\lesssim\min_{j=1,2,\ldots,L}|y_j|^{-\frac{1}{k_j+1}},\]
provided $k_m \geq L-m+1$.
\end{theorem}

The structural analysis of polynomial roots in this part is expected to hold with respect to polynomials over fields other than $\RR$. This corresponds with the arguments in \cite{kowWri}, which are presented for non-Archimedean fields but also hold over $\RR$. Wright has developed a framework for the study of oscillatory integrals over fields other than $\RR$ and proved analogues of the Phong-Stein cluster bound for such oscillatory integrals. Using bounds for oscillatory integrals over $\CC$ from \cite{wrightCompOscInt}, one could obtain a complex analogue of Theorem \ref{thm:oscIntEstFine}; the proof would be the same as the proof of Theorem \ref{thm:oscIntEstFine}, as the root structure we make use of depends only on the size of the coefficients.

\begin{theorem}We consider complex oscillatory integrals
\[I(x,y)=\int_{\CC} e(\Phi(t))\phi(t)dt,\]
with $\phi\in C_c^{\infty}(\CC)$ and $e(z)=e^{i(\Re(z)+\Im(z))}$. The phases $\Phi$ we consider are given by \[\Phi(z)=xz+\frac{y_1}{k_1+1}z^{k_1+1}+\ldots+\frac{y_L}{k_L+1}z^{k_L+1},\]
with $x,y_1,\ldots,y_L\in\CC$ and $y_L\neq 0$. We have that \begin{equation}\label{eq:compOscIntBound}\left|I(x,y)\right|\lesssim\min_{j=1,2,\ldots,L}|y_j|^{-\frac{2}{k_j+1}},\end{equation}
provided $k_1\geq L$.
More generally, if $\max_j|y_j|^{\frac{1}{k_j}}=|y_m|^\frac{1}{k_m}$, then \[\left|I(x,y)\right|\lesssim\min_{j=1,2,\ldots,L}|y_j|^{-\frac{2}{k_j+1}},\]
provided $k_m \geq L-m+1$.
\end{theorem}
\begin{remark}
Note that we have the exponents $\frac{2}{k_j+1}$ appearing on the right hand side of \eqref{eq:compOscIntBound}, in distinction to the real case. This is a natural feature of complex oscillatory integrals, as examples in \cite{wrightCompOscInt} show.
\end{remark}

The primary structural result, Theorem \ref{thm:tierStruc}, is not framed for explicit computation. For example, it makes reference to the \emph{relative} size of roots rather than their \emph{actual} size. For applications, explicit size estimates are desirable. Indeed, we outline a procedure for estimating the size of roots in Lemma \ref{lem:heightEstProcedure} and it is this procedure which allows us to strengthen the statement of Theorem \ref{thm:tierStruc} to the refined structural result, Theorem \ref{thm:tierStrucFine}. 

There is one final component of our structural investigations. Our analysis for tier stratification suggests the consideration of a factorised polynomial expression, with distinct factors corresponding to distinct root tiers. This is because of the way in which different coefficients are substantial in our analysis of different root tiers. For monic polynomials, $\Psi$, we analyse such a factorisation in Section \ref{sec:nearFact}. We find that the factorisation is quantifiably close to the original polynomial $\Psi$. Furthermore, the root structure of $\Psi$ is essentially preserved by its  rough factorisation.
\begin{theorem*}
We consider monic polynomials 
\[x+y_1t^{k_1}+\ldots +y_{L-1}t^{k_{L-1}}+t^{k_L}.\]
There exists a polynomial $\widetilde{\Psi}=\prod_{l=1}^s\widetilde{\Psi}_l(t)$ which  roughly factorises $\Psi$ in the following sense.

The polynomial $\widetilde{\Psi}_l$ has roots, $\widetilde{\mathcal{T}}_l$, which are all of mutually comparable magnitude. Furthermore, for any choice of roots $w_j\in \widetilde{\mathcal{T}}_j$, we have that 
\[|w_1|\gg|w_2|\gg\ldots\gg|w_s|.\] 

There exists a covering, $N(\mathcal{R})$, of the roots, $\mathcal{R}\subset \CC$, of $\Psi$ which satisfies the following. Each connected component of $N(\mathcal{R})$, which we call a cell, is given by a ball. Each root $w$ of $\Psi$ belongs to a unique cell, which has radius $\ll \epsilon |w|$, with $\epsilon$ as in Theorem \ref{thm:tierStruc}. Cells are strongly separated in the sense that, for distinct cells $B$ and $B'$, $d(B,B')\gg \max\lbrace\operatorname{diam}B, \operatorname{diam}B'\rbrace$.

Roots of $\widetilde{\Psi}$ are close to the roots of $\Psi$ in the following sense. For a cell $B$ containing exactly $m$ roots of $\Psi$, $B$ contains exactly $m$ roots of $\widetilde{\Psi}$, and we also have that $1\leq m\leq L$.

The rough factorisation  $\widetilde{\Psi}$ is close to $\Psi$ in the following sense. For $t\notin N(\mathcal{R})$,
\begin{equation}\label{eq:approxFactError}|\Psi(t)-\widetilde{\Psi}(t)|\ll | \Psi(t)|.\end{equation}
\end{theorem*}
\begin{remark}The constant in \eqref{eq:approxFactError} can be made arbitrarily small provided the tier regime, which we define in Section \ref{sec:strata}, is specified with sufficiently strong separation between tiers. More concretely, provided we make suitable restrictions on the coefficients, $x,y_1,\ldots,y_L$, we can ensure the rough factorisation $\widetilde{\Psi}$ is distinct from $\Psi$ and arbitrarily close to it. The conditions on the coefficients required for such a close approximation may be realised according with our height estimation lemma, Lemma \ref{lem:heightEstProcedure}.
\end{remark}

\section{Overview} The critical observation that forms the foundation of this paper is that many of the symmetric functions of the roots of $\Psi$ are vanishing. Those that are sufficiently far from vanishing determine the size of the roots. Additionally, they tell us how many roots might cluster about a point. We denote the roots of $\Psi$ by $\mathcal{R}$, they may appear with multiplicity and for this reason we follow a multi-set convention specified below. For at most $L$ non-zero critical indices $j\in \left\lbrace  D_1,D_2,\ldots,D_L\right\rbrace$, we have that \[S_{j}(\mathcal{R})=\sum_{\substack{|\mathcal{S}|=j\\ \mathcal{S}\subset \mathcal{R}}}\prod_{z\in\mathcal{S}}(-z)\neq 0.\] In particular, these indices are given by $D_1=k_{L}-k_{L-1}$, $D_2=k_{L}-k_{L-2}$, $\ldots$, $D_L=k_L-k_0=k_L$. Throughout, we work with reference to these symmetric functions, as well as those which are vanishing.

In Section \ref{sec:RootNotation}, we introduce notation. In Section \ref{sec:strucModel}, we present a model example for root structure and its relation to the polynomial coefficients (corresponding to the symmetric functions). 

The core of our structural analysis is developed in Part \ref{part:ImpRootStruc}. Indeed, Sections \ref{sec:strata}, \ref{sec:strucInTier}, and \ref{sec:seriesExp} contain the proof of the main structure Theorem \ref{thm:tierStruc}. This analysis pays no great heed to the polynomial coefficients, $x,y_1,\ldots,y_L$. Indeed, the structural results we obtain in Part \ref{part:ImpRootStruc} depend implicitly on the coefficients.

In Section \ref{sec:strata}, we use the vanishing of certain $S_j$ to provide a stratification of the roots into tiers and formulate the symmetric equations with respect to these tiers. 

In Section \ref{sec:strucInTier} we characterise how clustering might occur \emph{within} a given tier $\mathcal{T}$. Here, we analyse the simultaneous (near) vanishing of particular $S_j(\mathcal{T})$ and determine when this is inconsistent with many roots being close together. For this part of the analysis, we form a series expansion of particular $S_j(\mathcal{T})$ in terms of highlighted roots, which we suppose are close together. Section \ref{sec:seriesExp} contains the derivation of the distinguished root expansion that we use to prove Theorem \ref{thm:singleTierStruc}, which concerns the root structure in a single tier.

In Part \ref{part:ExpRootStruc}, we build on the tools developed in Part \ref{part:ImpRootStruc} to reveal more explicit root structure in specific instances. In Section \ref{sec:heightEst}, we outline an algorithm for estimating the heights of root tiers. This height estimation procedure can then be applied to polynomials satisfying the hypotheses of the refined structure Theorem \ref{thm:tierStrucFine}. The result we thus obtain feeds directly into our previous work to give Theorem \ref{thm:tierStrucFine} as a corollary. Section \ref{sec:nearFact} provides an explicit  rough factorisation, $\widetilde{\Psi}=\prod_{l=1}^s\widetilde{\Psi}_l(t)$, of monic polynomials $\Psi$. We show that the roots of $ \Psi$ and the roots of $\widetilde{\Psi}$ almost coincide.

Part \ref{part:oscInt} contains a proof of Theorem \ref{thm:oscIntEstFine} and presents some more refined oscillatory integral estimates. We also give examples for the sharpness of some of the oscillatory integral estimates, these examples are also examples of the sharpness of the structure theorems. 

\textbf{Acknowledgements.} 
The author was supported by The Maxwell Institute Graduate School in Analysis and its Applications, a Centre for Doctoral Training funded by the UK Engineering and Physical Sciences Research Council (Grant EP/L016508/01), the Scottish Funding Council, Heriot-Watt University and the University of Edinburgh. This work was first presented as part of the author's thesis, \cite{mythesis}.

The author would like to thank Jim Wright for the patient introduction of the problem and many helpful conversations at the challenging initial stages of remote working. The author would like to thank Kevin Hughes for some helpful suggestions on the presentation of this work. The author would like to thank Andrew Clausen for the signaling potential connections in the root-finding literature. The author also wishes to shout out Nammy Wams.

\section{Notation}\label{sec:RootNotation}
We here introduce some of the notation that we will be using throughout this paper. We use $\varepsilon$ throughout, often with indexing subscripts, for error terms that appear in the analysis, where we can have suitable control on their size. Any $\varepsilon$ that appears will be the sum of (signed) products of $k$ roots, for some $k$, with size bounds adapted to an appropriate scale. 

To account for repeated roots, it should be understood that we are working with multi-sets. For example, $\lbrace 1,1,1\rbrace$ should be considered a $3$-element (multi-)set. One would more formally write this multi-set as $\lbrace 1^{(1)},1^{(2)},1^{(3)}\rbrace$, indexing set elements by their multiplicity, so we can properly speak about distinct set elements. Another example is the fundamental theorem of algebra, which may be expressed as follows. If $\Psi$ is a degree $k_{L}$ polynomial, then the (multi-)set of roots of $\Psi$, which we denote by $\mathcal{R}$, contains $k_{L}$ elements.

We use $\mathcal{R}$ to denote the roots of $\Psi$. We also use $\mathcal{C}$, $\mathcal{T}$, and $\mathcal{S}$  to denote appropriate subcollections of roots. We use $\mathcal{K}$ and $\mathcal{D}$ to denote sets of integer indices, these will be specified but should be thought of as the exponents of terms in $\Psi$ and the differences between these exponents. 

Throughout, we fix some ordering of the roots of $\Psi(t)=x+y_1t^{k_1}+y_2 t^{k_2}+\ldots+y_Lt^{k_L}$:
\begin{equation}\label{eq:rootOrderDef}|z_1|\geq |z_2|\geq \ldots\geq |z_{k_{L}}|.\end{equation}
Later, we will use the notation $w_1,w_2,\ldots,w_{k_{L}}$ when we wish to take an arbitrary enumeration of the roots of $\Psi$.

Throughout this document $C$ will be used to denote a constant, its value may change from line to line. We use the notation $X\lesssim Y$ or $Y\gtrsim X$ if there exists some implicit constant $C$ such that $X\leq CY$. When we wish to highlight the dependence of the implied constant $C$ on some other parameter, say $C=C(M)$, we will use the notation $X\lesssim_M Y$. We use the notation $X\ll Y$ or $Y\gg X$ if there exists some suitable large constant $D$ such that $DX\leq Y$.

\begin{definition}Throughout, we consider elementary symmetric functions of (a subset of) roots of $\Psi$, \[S_j(\mathcal{A})=\sum_{\mathcal{S}\subset \mathcal{A},|\mathcal{S}|=j}\prod_{z\in \mathcal{S}}(-z),\]
where $\mathcal{A}$ is some subset of the roots of $\Psi$.
\end{definition}

\section{A model example for root structure}\label{sec:strucModel}
Working by example, we now give an indication of the possible root structure of a particular $\Psi$ and see how this relates to the coefficients. We consider $\Psi$ with \[\Psi(t)=y_2\prod_{j=1}^2\left(t^{k_1}-\alpha_j^{k_1}\right).\]
Note that all real polynomials $x+y_1t^{k_1}+ y_2t^{2k_1}$ can be expressed in this way. The roots of the polynomial $\Psi$ are easily recognised: they appear as the $k_1$th roots of $\alpha_1^{k_1}$ and $\alpha_2^{k_1}$. There are some qualitatively different scenarios for the structure of these roots. These depend on the relative size of $\alpha_1$ and $\alpha_2$. They also depend on the cancellation between $\alpha_1^{k_1}$ and $\alpha_2^{k_1}$. Without loss of generality, suppose that $|\alpha_1|\geq|\alpha_2|$. In the case of positive $\alpha_1$ and $\alpha_2$ with $k_1=9$, the roots of $\Psi$ are sketched in Figure \ref{fig:psiRootsArgand}. Throughout the remainder of this section, $\epsilon$ is some suitably small fixed parameter.

\begin{figure}[ht]
\includegraphics[width=0.6\textwidth]{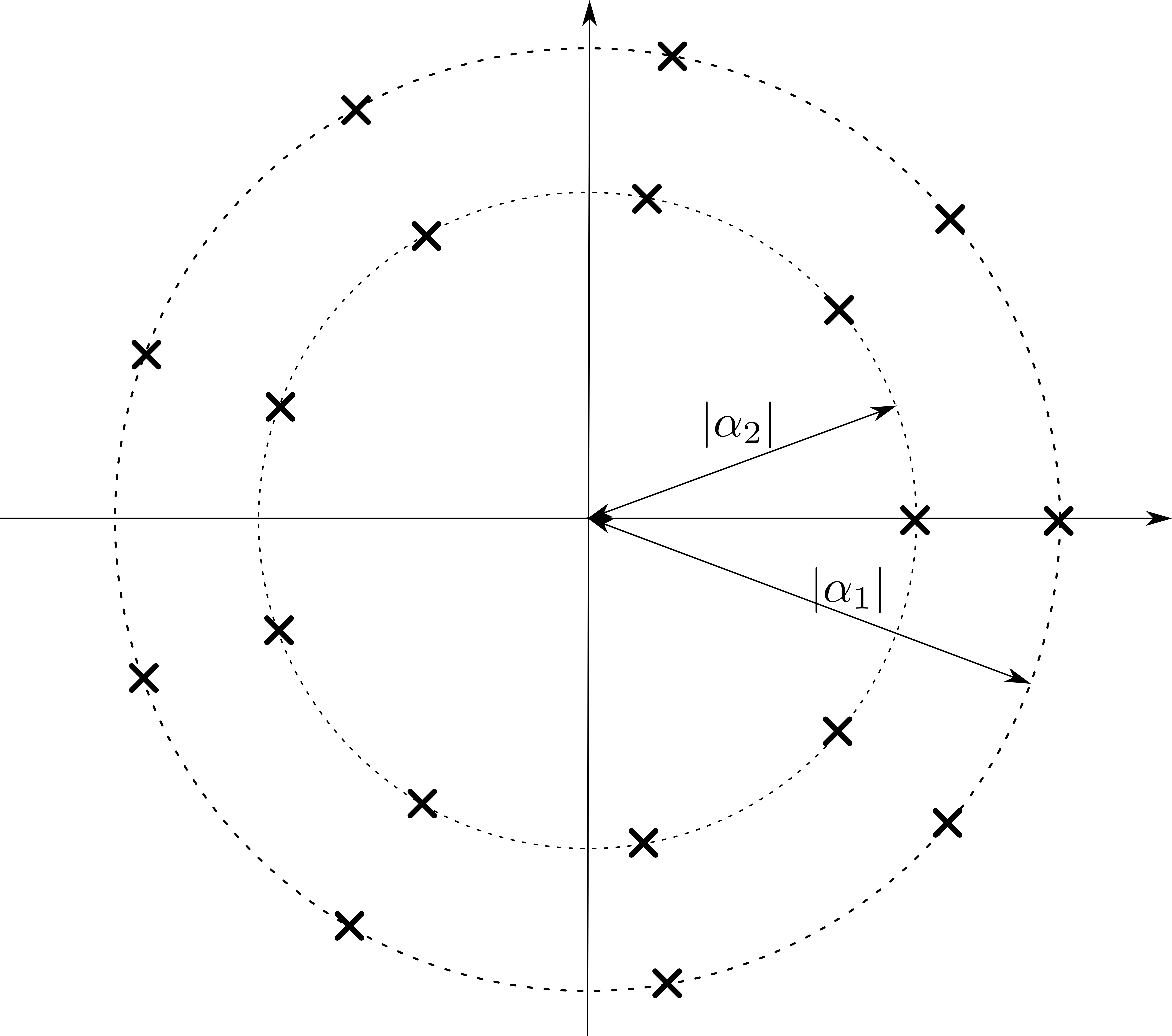}
\centering
\caption{The roots of $\Psi$ for $k_1=9$ and positive $\alpha_j$.}\label{fig:psiRootsArgand}
\end{figure}

Let us first consider the case where \underline{$\bm{\alpha_2=0}$ \textbf{and} $\bm{\alpha_1\neq 0}$}. Here, there are \textbf{two tiers of roots}. There are $k_1$ repeated $0$ roots and if we divide out the corresponding factor $t^{k_1}$ from $\Psi$ we are left with the polynomial $y_2(t^{k_1}-\alpha_1^{k_1})$, from which the location of the non-zero roots can be observed directly as the $k_1$th roots of $-S_{k_1}(\mathcal{R})=-\frac{y_1}{y_2}=\alpha_1^{k_1}$. The non-zero roots are in $\mathcal{T}_1$ and the tier $\mathcal{T}_2$ consists of all zero roots. In this case, for roots $w\in\mathcal{T}_1$, $\bm{B(w,\epsilon|w|)}$ \textbf{contains only the root }$\bm{w}$. Later on, being able to decouple equations for large and small roots in a similar fashion will critically allow us to analyse the structure of roots in distinct tiers. 

The \underline{\textbf{case that} $\bm{\alpha_1^{k_1}}$ \textbf{is close to }$\bm{-\alpha_2^{k_1}}$}, in particular, $|\alpha_1^{k_1}+\alpha_2^{k_1}|\ll |\alpha_1^{k_1}|$. The \textbf{roots of} $\bm{\Psi}$ \textbf{appear in one tier} and they are close to being the $2k_1$th roots of $-\alpha_1^{k_1}\alpha_2^{k_1}$: there are $2k_1=k_2$ roots $w_{j}$  with $|w_{j}|\sim |\alpha_1|$---such roots are in tier $\mathcal{T}_1$. Within $\mathcal{T}_1$ the roots are separated in the sense that, for roots $w\in\mathcal{T}_1$, $\bm{B(w,\epsilon|w|)}$ \textbf{contains only the root }$\bm{w}$. Note that the coefficients of $t^0$, $t^{k_1}$, and $t^{2k_1}$ in $\Psi$ reflect this behaviour in the fact that $\left|\frac{x}{y_2}\right|^\frac{1}{2k_1}=\left|S_{2k_1}(\mathcal{R})\right|^\frac{1}{2k_1}=\left|\alpha_1^{k_1}\alpha_2^{k_1}\right|^\frac{1}{2k_1}\gg \left|\alpha_1^{k_1}+\alpha_2^{k_1}\right|^\frac{1}{k_1}=\left|S_{k_1}(\mathcal{R})\right|^\frac{1}{k_1}=\left|\frac{y_1}{y_2}\right|^\frac{1}{k_1}$. 

The \underline{\textbf{case where} $\bm{|\alpha_1|}$ \textbf{is comparable to} $\bm{|\alpha_2|}$} but the $t^{k_1}$ coefficient of $\Psi$, $-y_2\left(\alpha_1^{k_1}+\alpha_2^{k_1}\right)$, does not display significant cancellation, i.e. $|\alpha_1^{k_1}+\alpha_2^{k_1}|\sim |\alpha_1^{k_1}|$. This scenario has \textbf{one tier of roots}, $\mathcal{T}_1=\lbrace z_1,z_2,\ldots,z_{k_2}\rbrace$. Here, all the roots are of comparable size and $\bm{B(w,\epsilon|w|)}$\textbf{ contains at most two roots for any root }$\bm{w\in\mathcal{T}_1$}. Two roots which might get close are those roots which share the same argument if $\alpha_1^{k_1}$ and $\alpha_2^{k_1}$ are real valued with the same sign. The coefficients of $t^0$, $t^{k_1}$, and $t^{2k_1}$ in $\Psi$ reflect this behaviour in the fact that \[\left|\frac{x}{y_2}\right|^\frac{1}{2k_1}=\left|S_{2k_1}(\mathcal{R})\right|^\frac{1}{2k_1}=\left|\alpha_1^{k_1}\alpha_2^{k_1}\right|^\frac{1}{2k_1}\sim \left|\alpha_1^{k_1}+\alpha_2^{k_1}\right|^\frac{1}{k_1}=\left|\frac{y_1}{y_2}\right|^\frac{1}{k_1}=\left|S_{k_1}(\mathcal{R})\right|^\frac{1}{k_1}.\]

Finally, we consider the \underline{\textbf{case where} $\bm{0\neq |\alpha_2|\ll |\alpha_1|}$}. Here roots of $\Psi$ appear in \textbf{two tiers}: there are $k_1$ roots $w_{j,2}$  with $|w_{j,2}|=|\alpha_2|$---such roots are in tier $\mathcal{T}_2$---and the remaining $k_1$ roots $w_{l,1}$ satisfy $|w_{l,1}|=|\alpha_1|$---such roots are in tier $\mathcal{T}_1$. Within each tier the roots are separated in the sense that, $\bm{B(w,\epsilon|w|)}$\textbf{ contains only one root for any root }$\bm{w\in\mathcal{T}_j}$. 
We also have separation between tiers: given roots $w_{1},\in\mathcal{T}_1$ and $w_{2}\in\mathcal{T}_2$, $|w_{1}|\gg |w_{2}|$. Note that the coefficients of $t^0$, $t^{k_1}$, and $t^{2k_1}$ in $\Psi$ reflect this behaviour in the fact that \[\left|S_{2k_1}(\mathcal{R})\right|^\frac{1}{2k_1}=\left|\frac{x}{y_2}\right|^\frac{1}{2k_1}=\left|\alpha_1\alpha_2\right|^\frac{1}{2}< \left|\frac{y_1}{y_2}\right|^\frac{1}{k_1}=\left|S_{k_1}(\mathcal{R})\right|^\frac{1}{k_1}=\left|\alpha_1^{k_1}+\alpha_2^{k_1}\right|^\frac{1}{k_1}\sim |\alpha_1|\]
and 
\[\left|S_{k_1}(\mathcal{T}_2)\right|^\frac{1}{k_1}=|\alpha_2|\sim\left|\frac{\alpha_1^{k_1}\alpha_2^{k_1}}{\alpha_1^{k_1}+\alpha_2^{k_1}}\right|^\frac{1}{k_1}=\left|\frac{x}{y_1}\right|^\frac{1}{k_1}=\left|\frac{S_{2k_1}(\mathcal{R})}{S_{k_1}(\mathcal{R})}\right|^\frac{1}{k_1}\]
\[\ll \left|\frac{y_1}{y_2}\right|^\frac{1}{k_1}=\left|\alpha_1^{k_1}+\alpha_2^{k_1}\right|^\frac{1}{k_1}\sim |\alpha_1|=\left|S_{k_1}(\mathcal{T}_1)\right|^\frac{1}{k_1}.\]
These last equations correspond to the separation of certain height estimates. Such height estimates will not form a part of our initial structural analysis, but we eventually consider them more explicitly in Section \ref{sec:heightEst}.

\part{Implicit root structure}\label{part:ImpRootStruc}
In this part, we work to uncover some of the root structure of real polynomials
\[\Psi(t)=x+y_1t^{k_1}+\ldots+y_Lt^{k_L},\]
with exponents taken from a fixed set $\lbrace 0,k_1,\ldots,k_L\rbrace$ and $x,y_L\neq 0$. Throughout, we denote the set of roots of $\Psi$ by $\mathcal{R}$. Essentially, as there are at most $L+1$ non-vanishing coefficients, we will find that at most $L$ roots can coalesce about a point. Throughout this part, our analysis will be carried out with respect to particular reference heights, $h_1,h_2,\ldots,h_L$, which depend implicitly on a given polynomial $\Psi$. We postpone the discussion of more explicit tools for locating roots to Part \ref{part:ExpRootStruc}.

\section{Root tier stratification}\label{sec:strata}
In this section, we work to stratify the roots into tiers. The stratification of roots is suggested by the example in Section \ref{sec:strucModel}. With this in mind, we define certain reference heights $h_j$ for the roots. Before proceeding, let us introduce some useful indexing notation. 
\begin{definition}
 We set \begin{equation}\label{eq:dDef}  \begin{aligned}&d_0(\mathcal{R})=0,\,&& d_1(\mathcal{R})=k_{L}-k_{L-1},\,&&\; \ldots\;&& d_{L}(\mathcal{R})=k_1-k_0,\quad \text{and}\\ &D_0(\mathcal{R})=d_0(\mathcal{R}),\,&& D_1(\mathcal{R})=d_0(\mathcal{R})+d_1(\mathcal{R}),\,&&\; \ldots\;&& D_L(\mathcal{R})=d_1(\mathcal{R})+\ldots+d_L(\mathcal{R}).\end{aligned} \end{equation}

We also set $\mathcal{D}(\mathcal{R})=\left\lbrace 0, D_1(\mathcal{R}),\ldots,D_L(\mathcal{R})\right\rbrace$.
\end{definition}
Note that $D_j(\mathcal{R})=k_L-k_{L-j}$, although we expressed it slightly differently in the definition to emphasise that it is the sum of consecutive $d_i$.

Throughout this section, we will be working with reference to \emph{all} roots, $\mathcal{R}$, and so we suppress the argument of $D_j=D_j(\mathcal{R})$ and $d_j=d_j(\mathcal{R})$. Similarly, in this section, we set $\mathcal{D}=\mathcal{D}(\mathcal{R})$. We also write $d(j)=d_j$ and $D(j)=D_j$. 

\begin{definition}\label{def:hj}According with the root ordering, \eqref{eq:rootOrderDef}, we define the reference heights $h_1=|z_1|,h_2=|z_{D(1)+1}|,h_3=|z_{D(2)+1}|,\ldots, h_{L}=|z_{D(L-1)+1}|$. \end{definition}

The following Lemma \ref{lem:cancComp} shows exactly why our Definition \ref{def:hj} is a sensible one.

\begin{lemma}\label{lem:cancComp}Suppose that, for some $1\leq j\leq k_L$, $S_j(\mathcal{R})=0$. Then $|z_{j+1}|\sim |z_j|$.

As a consequence, since $S_j(\mathcal{R})=0$ for $j\notin\mathcal{D}=\lbrace 0,D_1,D_2,\ldots,D_L\rbrace$, we have that \begin{equation}\label{eq:trancheComp}\begin{aligned}&h_1&&=|z_1|&&\sim |z_2|&&\sim\ldots&&\sim |z_{D(1)}|,\\&h_2&&=|z_{D(1)+1}|&&\sim |z_{D(1)+2}|&&\sim\ldots&&\sim |z_{D(2)}|,\\&  && && && &&\vdots  &&\\ &h_L&&=|z_{D(L-1)+1}|&&\sim |z_{D(L-1)+2}|&&\sim\ldots&&\sim |z_{D(L)}|.\end{aligned}\end{equation}
\end{lemma}
\begin{proof}
Suppose that $S_j(\mathcal{R})=0$ and consider the corresponding root $z_j$. There are two cases to consider. In the case that $|z_j|=0$, the desired comparison follows directly from the inequality $|z_{j+1}|\leq |z_{j}|=0$.

It remains to consider the case that $z_j\neq 0$. The largest root outwith $\lbrace z_1,z_2,\ldots,z_j\rbrace$ is $z_{j+1}$. Additionally, the largest $j-1$ roots within $\lbrace z_1,z_2,\ldots,z_j\rbrace$ are $z_1,z_2,\ldots,z_{j-1}$. Thus we find that 
\[0=|S_j|\geq |(-z_1)(-z_2)\ldots(-z_j)|-\binom{k_L}{j}|z_1\ldots z_{j-1}z_{j+1}|.\]
Rearranging and dividing through by $|(-z_1)(-z_2)\ldots(-z_j)|$ gives the desired inequality.
\end{proof}

\begin{remark}
The constants of comparison we obtain in \eqref{eq:trancheComp} can be chosen to depend on the indices $k_1,k_2,\ldots,k_L$. \end{remark}

In the statement of Theorem \ref{thm:tierStruc}, we said that tiers of roots are well separated. It is thus natural, and in accordance with Lemma \ref{lem:cancComp}, to define the tiers of roots relative to the separation of the reference heights.
 \begin{definition}\label{def:tierDef}If we have that \[h_1\sim h_2\sim\ldots \sim h_{l(1)}\]\[\gg h_{l(1)+1}\sim \ldots \sim h_{l(1)+l(2)}\]\[\vdots\] \[\gg h_{l(1)+\ldots+l(s-1)+1}\sim \ldots \sim h_{l(1)+\ldots+l(s)}=h_L,\]
then we define the tiers as follows. First, let $l(0)=0$ and $L(i)=l(0)+l(1)+\ldots+l(i)$. Then we set \[\mathcal{T}_i= \left\lbrace z_{D(L(i-1))+1},z_{D(L(i-1))+2}\ldots, z_{D(L(i))}\right \rbrace.\]
We also define the reference height for the tiers by
\[h(\mathcal{T}_r)=h_{L(r-1)+1}.\]
 \end{definition}
This definition establishes the first part of our structure theorem, Theorem \ref{thm:tierStruc}. Indeed, by Lemma \ref{lem:cancComp}, we have that, for any choice of $w_i\in \mathcal{T}_i$
\[|w_1|\gg |w_2|\gg\ldots\gg |w_s|.\]

\begin{remark}
The choice of separation constants defining the tier regime must be suitably strong. Having a well separated tier regime will ensure we have suitable control on error terms. In this section, up to an error term, we derive explicit equations for the symmetric functions of roots in a given tier. The error terms must be small enough to feed into our later Theorem \ref{thm:singleTierStruc}. For the  rough factorisation of monic $\Psi$, Theorem \ref{thm:approxFact}, the definition of a tier regime requires much stronger separation, as we will see in Section \ref{sec:nearFact}.\end{remark}

Analogous to Definition \ref{eq:dDef}, it will be useful to have distinguished indices for each tier. These distinguished indices will later allow us to pick out critical symmetric functions of roots within each tier.

\begin{definition}\label{def:distIndices}
In the tier $\mathcal{T}_r$, we set  \begin{equation} \begin{split} d_0(\mathcal{T}_r)=0,\, d_1(\mathcal{T}_r)=d_{L(r-1)+1},\, d_2(\mathcal{T}_r)=d_{L(r-1)+2},\; \ldots\; d_{l(r)}(\mathcal{T}_r)=d_{L(r)},\, \text{and}\\D_0(\mathcal{T}_r)=0,\, D_1(\mathcal{T}_r)=d_0(\mathcal{T}_r)+d_1(\mathcal{T}_r),\, \ldots\; D_{l(r)}(\mathcal{T}_r)=d_1(\mathcal{T}_r)+\ldots+d_{l(r)}(\mathcal{T}_r).\; \end{split} \end{equation}
There is an appropriate $L(\mathcal{T}_r)=l(r)=L(r)-L(r-1)$.
 Note that $D_{l(r)}(\mathcal{T}_r)=|\mathcal{T}_r|=D(L(r))-D(L(r-1))$, and we also denote this by $D(\mathcal{T}_r)$.

If $\mathcal{T}=\mathcal{T}_r$, then the distinguished indices for the symmetric functions in $\mathcal{T}$ are given by $\mathcal{D}(\mathcal{T})=\left\lbrace 0, D_1(\mathcal{T}),\ldots, D_{L( \mathcal{T})} (\mathcal{T})\right \rbrace$.
\end{definition}

We now present the main result of this section, this lemma characterises the critical symmetric functions with respect to roots within each tier.

\begin{lemma}\label{lem:symEqWithinTier}According with Definition \ref{def:tierDef}, fix a tier $\mathcal{T}=\mathcal{T}_r$. 

For the positive critical indices $j\in \mathcal{D}(\mathcal{T})$, \[S_{j}(\mathcal{T})=c_j(\mathcal{T})+\varepsilon_{j}(\mathcal{T}),\]
where $|\varepsilon_{j}(\mathcal{T})|\ll h(\mathcal{T})^j$ and \[c_j(\mathcal{T})=\frac{S_{D(L(r-1))+j}(\mathcal{R})}{S_{D(L(r-1))}(\mathcal{T}_1\cup\mathcal{T}_1\cup\ldots \mathcal{T}_{r-1})}.\]

For  $j\notin \mathcal{D}(\mathcal{T})$, with $1\leq j\leq |\mathcal{T}|$, we have that 
\[S_{j}(\mathcal{T})=\varepsilon_{j}(\mathcal{T}),\]
where $|\varepsilon_{j}(\mathcal{T})|\ll h(\mathcal{T})^j$.
\end{lemma}

\begin{proof}We consider the symmetric functions of order $D(L(r-1))+j$ for $1\leq j\leq D(\mathcal{T})$. For ease of notation, we set $D=D(L(r-1))$ and we denote the reference height $h(\mathcal{T})=h_{L(r-1)+1}$ by $h$. 

Splitting the sum defining symmetric functions according to the size of the summands, we have that
\begin{equation}\label{eq:symEqForTier}S_{D+j}(\mathcal{R})=S_D(\mathcal{T}_1\cup\mathcal{T}_2\cup\ldots\cup\mathcal{T}_{r-1})S_j(\mathcal{T})+ \sum_{|\mathcal{S}'|=D+j}\prod_{z\in \mathcal{S}'}(-z),\end{equation}
where the sum in $\mathcal{S}'\subset\mathcal{R}$ is taken over $\mathcal{S}'\cap\left(\mathcal{T}_1\cup\mathcal{T}_2\cup\ldots \mathcal{T}_{r-1}\cup \mathcal{T}_{r}\right)^c\neq \emptyset$. Any $\mathcal{S}'$ appearing in this proof should be understood as subject to these restrictions. Note that \[|S_D(\mathcal{T}_1\cup\mathcal{T}_2\cup\ldots\cup\mathcal{T}_{r-1})|=|z_1z_2\ldots z_{D}|.\] 

Every product appearing in \[\sum_{|\mathcal{S}'|=D+j}\prod_{z\in \mathcal{S}'}(-z)\]
contains at least one root, $z$, outwith $\mathcal{T}_1\cup\mathcal{T}_2\cup\ldots \cup\mathcal{T}_{r}$ and, for any such $z$, $|z|\ll h$. The largest $D+j-1$ roots are $z_1,z_2,\ldots, z_{D+j-1}$. Thus we see that  \[\left|\sum_{|\mathcal{S}'|=D+j}\prod_{z\in \mathcal{S}'}(-z)\right|\ll \left|z_1z_2\ldots z_{D+j-1} h\right|.\] Normalising and rearranging \eqref{eq:symEqForTier} we thus see that \[S_j(\mathcal{T})=  S_D(\mathcal{T}_1\cup\mathcal{T}_1\cup\ldots\cup\mathcal{T}_{r-1})^{-1}S_{D+j}(\mathcal{R})+ \varepsilon_{j}(\mathcal{T}),\]
where \[\varepsilon_{j}(\mathcal{T})=-S_D(\mathcal{T}_1\cup\mathcal{T}_1\cup\ldots\cup\mathcal{T}_{r-1})^{-1}\sum_{|\mathcal{S}'|=D+j}\prod_{z\in \mathcal{S}'}(-z).\] To conclude, we observe that \[\left|\varepsilon_{j}(\mathcal{T})\right|\ll \left|z_{D+1}z_{D+2}\ldots z_{D+j-1}h\right|\sim h^j,\]
as required.
\end{proof}

\section{Root structure within tiers}\label{sec:strucInTier}
Recall how we defined the tiers of roots in Definition \ref{def:tierDef}, which immediately gives the separation between tiers and part of Theorem \ref{thm:tierStruc}. In the regime with tiers $\mathcal{T}_1,\mathcal{T}_2,\ldots,\mathcal{T}_{s}$, we have that 
\[h(\mathcal{T}_1)\gg h(\mathcal{T}_2)\gg\ldots \gg h(\mathcal{T}_s),\]
where $h(\mathcal{T}_r)=\max_{w\in\mathcal{T}_r}|w|$. To complete the proof of Theorem \ref{thm:tierStruc}, we must establish the structure of roots \emph{within} a tier. In this section, we prove Theorem \ref{thm:singleTierStruc}, which, combined with Lemma \ref{lem:symEqWithinTier}, gives the structure Theorem \ref{thm:tierStruc}.

Let us fix $\mathcal{T}=\mathcal{T}_r$ for some $r$, where we have that $h(\mathcal{T})=h_{L(r-1)+1}>0$. For ease of notation, throughout this section, we denote by $h$ the reference height for roots in $\mathcal{T}_r$, $h=h(\mathcal{T}_r)=h_{L(r-1)+1}$.
With care, we can leverage the statement of Lemma \ref{lem:symEqWithinTier} to determine how roots in $\mathcal{T}$ may cluster about a point. 

Recall from Lemma \ref{lem:symEqWithinTier} that, for $j\notin\mathcal{D}(\mathcal{T})$, $S_j(\mathcal{T})$ is near vanishing. The following theorem tells us what kind kind of clustering can occur according to how many symmetric functions of $\mathcal{T}$ are far from vanishing. In particular, we consider what happens when the symmetric functions, $S_j(\mathcal{T})$, are near vanishing away from some (unspecified) set of exceptional indices $\widetilde{\mathcal{D}}(\mathcal{T})\subset \mathcal{D}(\mathcal{T})$.

\begin{theorem}\label{thm:singleTierStruc}Suppose that, for $0\leq j \leq D(\mathcal{T})$, \[S_j(\mathcal{T})=\varepsilon_j(\mathcal{T}),\text{ for }j\notin \widetilde{\mathcal{D}}(\mathcal{T}),\] where $|\varepsilon_j(\mathcal{T})|\ll h^j$ and $\widetilde{\mathcal{D}}(\mathcal{T})\subset \mathcal{D}(\mathcal{T})$. Then at most $\tilde{L}(\mathcal{T})\coloneqq|\widetilde{\mathcal{D}}(\mathcal{T})|-1\leq |\mathcal{D}(\mathcal{T})|-1=L(\mathcal{T})$ roots in $\mathcal{T}$ can cluster about a point. More precisely, there is is some suitably small $\epsilon>0$ such that, for any root $w\in\mathcal{T}$, we have $|w|\sim h>0$ and $B(w,\epsilon h)$ contains at most $\tilde{L}(\mathcal{T})$ roots $w_i\in\mathcal{T}$.
\end{theorem}

We have as a corollary of Theorem \ref{thm:singleTierStruc} and Lemma \ref{lem:symEqWithinTier} the following structure theorem, of which Theorem \ref{thm:tierStruc} is a special case.
\begin{theorem}\label{thm:tierStrucGivenRegime}
We fix a set of exponents $0=k_0< k_1<k_2<\ldots<k_L$ and consider real polynomials whose exponents are drawn from this set. For any real polynomial $\Psi(t)=x+y_1t^{k_1}+y_2t^{k_2}+\ldots+y_Lt^{k_L}$ with $x,y_L\neq 0$, we have the following.

We suppose that the polynomial coefficients $(x,y)$ are such that we are in the tier regime indexed by $(l_1,l_2,\ldots,l_s)$ in Definition \ref{def:tierDef}.
 
The roots of $\Psi$ are stratified into $s$ tiers of roots $\mathcal{T}_1,\mathcal{T}_2,\ldots ,\mathcal{T}_s$, where $1\leq s \leq L$. The tiers are separated in the sense that, if we take $w_{i}\in \mathcal{T}_i$, then 
\[|w_{1}|\gg |w_{2}|\gg\ldots \gg |w_{s}|.\] 
 
 At most $L(\mathcal{T}_r)=l_r\leq L$ roots can cluster about a root: there is some suitable small parameter $\epsilon$ such that, for all $ w\in\mathcal{T}_{r}$, there are at most $L(\mathcal{T}_r)=l_r$ roots of $\Psi$ in $B(w,\epsilon|w|)$.
\end{theorem}

 We have expressed Theorem \ref{thm:singleTierStruc} in a slightly more general form than is required to prove Theorem \ref{thm:tierStruc}. As a black box, Theorem \ref{thm:singleTierStruc} can give improvements to our main structure Theorem \ref{thm:tierStruc}. Indeed, if it is applied with reference to our later Proposition \ref{prop:tierRefined}, we can obtain \ref{thm:tierStrucFine}. It is for this reason we have framed the theorem in terms of critical indices $\widetilde{\mathcal{D}}$, rather than $\mathcal{D}$, since some of the symmetric functions indexed by $\mathcal{D}$ might still be close to vanishing. Nevertheless, all of the tildes appearing in this section can safely be ignored on first reading under the assumption that there is only one tier as all of the essential ideas are contained in this case.

In this section, roots in $\mathcal{T}$ \emph{are not} ordered in terms of size: $w_1,w_2,\ldots,w_{D(\mathcal{T})}$ is some enumeration of the roots in $\mathcal{T}$. Our analysis works by highlighting some of these roots, $\mathpzc{h}\subset\mathcal{T}$, which we will assume are close together. We expand the symmetric functions in terms of these highlighted roots. We will recover some structural statements about the roots from the highlighted expansions. We denote the excluded roots by $\mathpzc{e}=\mathcal{T}\backslash \mathpzc{h}$.

First, we state the following lemma, which is verified at a glance.
\begin{lemma}\label{lem:serHighlight}Let $\mathpzc{h}\subset \mathcal{T}$ with $|\mathpzc{h}|=m$. Then
\[S_{D(\mathcal{T})-m}(\mathcal{T})=\sum_{l=0}^{\min\lbrace m, D(\mathcal{T})-m\rbrace}S_{l}(\mathpzc{h})S_{D(\mathcal{T})-m-l}(\mathpzc{e}).\]
\end{lemma}

Since we will be investigating highlighted roots that are close together, we distinguish one of these roots to further refine the expansion. By a suitable recursive procedure, which is outlined in Section \ref{sec:seriesExp}, we can apply Lemma \ref{lem:serHighlight} to obtain the following. 

\begin{lemma}
\label{lem:serDistinguished}
Set \[a_m(j)=(-1)^{j-1}\binom{m-1+j}{m-1}.\]
Suppose we have $m$ highlighted roots $\mathpzc{h}=\left\lbrace w_1,w_2,\ldots,w_m\right\rbrace\subset\mathcal{T}$. Then, 
\[S_{D(\mathcal{T})-m}(\mathcal{T})=(-w_{m+1})(-w_{m+2})\ldots(-w_{D(\mathcal{T})})\]\[+\sum_{j=1}^{D(\mathcal{T})-m}a_m(j)(-w_1)^jS_{D(\mathcal{T})-m-j}(\mathcal{T})+\varepsilon_{D(\mathcal{T})-m},\]
where \[|\varepsilon_{D(\mathcal{T})-m}|\lesssim \max_{w,w'\in\mathpzc{h}}{|w-w'|}h^{D(\mathcal{T})-m-1},\] if $m\geq 2$ and $\varepsilon_{D(\mathcal{T})-m}=0$ if $m=1$.
\end{lemma}
\begin{remark}
It may appear that there is an error in the statement of Lemma \ref{lem:serDistinguished}, due to the apparent double counting of certain expressions. However, these are accounted for in the error term. The reason we desire such a series expansion is because those expressions $S_{D(\mathcal{T})-m-j}(\mathcal{T})$ are much more explicit than, for example, $S_{D(\mathcal{T})-m-j}(\mathpzc{e})$. Indeed, we can relate $S_{D(\mathcal{T})-m-j}(\mathcal{T})$ to the coefficients of our original polynomial via Lemma \ref{lem:symEqWithinTier} and we have no such tools for $S_{D(\mathcal{T})-m-j}(\mathpzc{e})$ or $S_{j}(\mathpzc{h})$.
\end{remark}

Recall Definition \ref{def:distIndices}, the definition of distinguished indices for a tier, $\mathcal{D}(\mathcal{T})$. It will also be necessary to count backwards from $D(\mathcal{T})$ to $0$ and pick out corresponding critical symmetric functions of roots in the tier. We make the following further definition of distinguished exponents for a given tier $\mathcal{T}=\mathcal{T}_r$.
\begin{definition}\label{def:distIndicesK}
We set $k_0(\mathcal{T})=0$,
\begin{equation} k_1(\mathcal{T})=d_{L(r)},\,k_2(\mathcal{T})=d_{L(r)}+d_{L(r)-1},\,
\ldots\; k_{L(\mathcal{T})}(\mathcal{T})=d_{L(r)}+\ldots+d_{L(r-1)+1}.\end{equation}
Note that  $k_{L(\mathcal{T})}(\mathcal{T})=D(\mathcal{T})$. We denote by $\mathcal{K}(\mathcal{T})$ these exponents. Note that \[\mathcal{K}(\mathcal{T})=D(\mathcal{T})-\mathcal{D}(\mathcal{T})=\left\lbrace D(\mathcal{T}), D(\mathcal{T})-D_1(\mathcal{T}),\ldots, 0 \right\rbrace.\] 

According with Theorem \ref{thm:singleTierStruc}, if we have $\tilde{L}(\mathcal{T})+1$ distinguished indices $0=\tilde{D}_0(\mathcal{T})<\tilde{D}_1(\mathcal{T})<\ldots< \tilde{D}_{\tilde{L}(\mathcal{T})}(\mathcal{T})=D(\mathcal{T})$, given as $\widetilde{\mathcal{D}}(\mathcal{T})\subset \mathcal{D}(\mathcal{T})$, then, corresponding with the above, we set $\widetilde{\mathcal{K}}(\mathcal{T})=D(\mathcal{T})-\widetilde{\mathcal{D}}(\mathcal{T})$. Naturally, we enumerate $\widetilde{\mathcal{K}}(\mathcal{T})$ by $0=\tilde{k}_0<\tilde{k}_1<\ldots<\tilde{k}_{\tilde{L}(\mathcal{T})}$.
\end{definition}

With this notation and our series expansion tools, we are now ready to carry out the analysis of how roots can cluster within $\mathcal{T}$.
\begin{proof}[Proof of Theorem \ref{thm:singleTierStruc}]
In this proof, we use the notation $D=D(\mathcal{T})=|\mathcal{T}|$, $\tilde{L}=\tilde{L}(\mathcal{T})=|\widetilde{\mathcal{D}}(\mathcal{T})|-1$, $\tilde{D}(j)=\tilde{D}_j(\mathcal{T})$, $\tilde{k}_j=\tilde{k}_j(\mathcal{T})$. 

We suppose, by way of contradiction, that, given some small $\epsilon>0$, there exist roots   $w_1,w_2,\ldots,w_{\tilde{L}+1}\in B(w_1,\epsilon h)\cap \mathcal{T}.$ Working with the symmetric functions $S_{D-1}(\mathcal{T}),$ $S_{D-2}(\mathcal{T}),$ $\ldots,$ $S_{D-(\tilde{L}+1)}(\mathcal{T})$ we will derive a system of equations which has no solution.

We can apply Lemma \ref{lem:serDistinguished} to obtain the following. For highlighted roots $\mathpzc{h}\subset \lbrace w_1,w_2,\ldots,w_{\tilde{L}+1}\rbrace \subset\mathcal{T}$ containing $m$ elements, we have
\[S_{D-m}(\mathcal{T})=(-w_{m+1})(-w_{m+2})\ldots(-w_D)+\sum_{j=1}^{D-m}a_m(j)(-w_1)^jS_{D-m-j}(\mathcal{T})+\varepsilon_{D-m}(\mathpzc{h}),\]
where $|\varepsilon_{D-m}(\mathpzc{h})|\ll h^{D-m}$. For $m+1\leq l\leq \tilde{L}+1$, replacing instances of $(-w_l)$ with $(-w_1)+((-w_l)-(-w_1))$ and observing that $|w_l-w_1|\leq \epsilon h$, we find that 
\[S_{D-m}(\mathcal{T})\]\[=(-w_1)^{\tilde{L}+1-m}(-w_{\tilde{L}+2})(-w_{\tilde{L}+3})\ldots(-w_D)+\sum_{j=1}^{D-m}a_m(j)(-w_1)^jS_{D-m-j}(\mathcal{T})+\varepsilon_{D-m}^{(1)},\]
where $|\varepsilon_{D-m}^{(1)}|\lesssim \epsilon h^{D-m}$.

Dividing through by $(-w_1)^{\tilde{L}+1-m}$, we obtain
\[(-w_1)^{m-(\tilde{L}+1)}S_{D-m}(\mathcal{T})\]\[=(-w_{\tilde{L}+2})\ldots(-w_D)+\sum_{j=1}^{D-m}a_m(j)(-w_1)^{j+m-(\tilde{L}+1)}S_{D-m-j}(\mathcal{T})+\varepsilon_{D-m}^{(2)},\]
where $|\varepsilon_{D-m}^{(2)}|\lesssim \epsilon h^{D-(\tilde{L}+1)}$.

Many of the terms appearing in the above sum are near vanishing. To pick out the significant terms, we observe $D-m-j\in \widetilde{\mathcal{D}}(\mathcal{T})$ requires that $j=D(\mathcal{T})-\tilde{D}_i(\mathcal{T})-m$ for $0\leq i\leq \tilde{L}$. The relevant set of indices $j$ is precisely the positive elements of $\widetilde{\mathcal{K}}(\mathcal{T})-m$. 
Thus, we find
\begin{equation}\label{eq:normalisedMExp}
\begin{split}&(-w_1)^{m-(\tilde{L}+1)}S_{D-m}(\mathcal{T})\\
=&(-w_{\tilde{L}+2})\ldots(-w_D)+\sum_{\substack{j\in \widetilde{\mathcal{K}}(\mathcal{T})-m\\j\geq 1}}^{D-m}a_m(j)(-w_1)^{j+m-(\tilde{L}+1)}S_{D-m-j}(\mathcal{T})\\+&\varepsilon_{D-m}^{(3)},\end{split}\end{equation}
where $\varepsilon_{D-m}^{(3)}$ is an error term with $|\varepsilon_{D-m}^{(3)}|\lesssim \epsilon h^{D-(\tilde{L}+1)}$. The normalised symmetric functions appearing in the sum are precisely $(-w_1)^{\tilde{k}_{i}-(\tilde{L}+1)}S_{\tilde{D}(\tilde{L}-i)}(\mathcal{T})$ where $\tilde{D}(\tilde{L}-i)\in \widetilde{\mathcal{D}}$ and $\tilde{D}(\tilde{L}-i)<D-m$. Observe that, from Lemma \ref{lem:serDistinguished}, $a_m(0)=(-1)$ so that, if $D-m\in\widetilde{\mathcal{D}}$, then \eqref{eq:normalisedMExp} can be rearranged to
\begin{equation}\label{eq:normalisedMExpRearranged}
\begin{split}\varepsilon_{D-m}^{(4)}&\\
=&(-w_{\tilde{L}+2})\ldots(-w_D)+\sum_{\substack{j\in \widetilde{\mathcal{K}}(\mathcal{T})-m\\j\geq 0}}^{D-m}a_m(j)(-w_1)^{j+m-(\tilde{L}+1)}S_{D-m-j}(\mathcal{T}),\end{split}\end{equation}
where $\varepsilon_{D-m}^{(4)}$ is an error term with $|\varepsilon_{D-m}^{(4)}|\ll h^{D-(\tilde{L}+1)}$.
In fact, \eqref{eq:normalisedMExpRearranged} is valid for all $m$ since, if $D-m\notin \widetilde{\mathcal{K}}$, then $|(-w_1)^{m-(\tilde{L}+1)}S_{D-m}(\mathcal{T})|\ll h^{D-(\tilde{L}+1)}$.

We set \[\vec{v}=\begin{pmatrix}
     (-w_{\tilde{L}+2})(-w_{\tilde{L}+3})\ldots(-w_D)\\
     (-w_1)^{\tilde{k}_{1}-(\tilde{L}+1)}S_{D(\tilde{L}-1)}\\
      \vdots \\
     (-w_1)^{\tilde{k}_{\tilde{L}}-(\tilde{L}+1)}S_{\tilde{D}(\tilde{L}-\tilde{L})}
\end{pmatrix},\quad \vec{\varepsilon}=\begin{pmatrix}
     \varepsilon_{D-1}^{(4)}\\
      \vdots \\
    \varepsilon_{D-(\tilde{L}+1)}^{(4)}
\end{pmatrix}.\] 
 Bringing the equations \eqref{eq:normalisedMExpRearranged} for $1\leq m \leq \tilde{L}+1$ together, we have a matrix equation
\begin{equation}\label{eq:normalisedMExpVec}\vec{\varepsilon}=M\vec{v},\end{equation}
with $M$ as specified below. Recall, from Lemma \ref{lem:serDistinguished}, that $a_m(j)=(-1)^{j-1}\binom{m-1+j}{m-1}$ so that $a_m(j-m)=(-1)^{j-m-1}\binom{j-1}{m-1}$. 

Let us first give an example of $M$. When $\tilde{L}=2$, and $\tilde{k}_1> 3$ we have \[  M=\left( {\begin{array}{ccc}
   1 & (-1)^{\tilde{k}_1-2} & (-1)^{\tilde{k}_2-2}\\
   1 & (-1)^{\tilde{k}_1-3}(\tilde{k}_1-1) & (-1)^{\tilde{k}_2-3}(\tilde{k}_2-1)\\
    1 & (-1)^{\tilde{k}_1-4}\frac{1}{2}(\tilde{k}_1-2)(\tilde{k}_1-1)  & (-1)^{\tilde{k}_2-4}\frac{1}{2}(\tilde{k}_2-2)(\tilde{k}_2-1) \\ \end{array} }\right).\]
    
Let us give an example of $M$ when it contains some zero entries. This happens, for example, if $\tilde{L}=2$ and $\tilde{k}_2>3$ but $\tilde{k}_1=2$. In this case, 
\[ M=\left( {\begin{array}{ccc}
   1 & (-1)^{\tilde{k}_1-2} & (-1)^{\tilde{k}_2-2}\\
   1 & -1 & (-1)^{\tilde{k}_2-3}(\tilde{k}_2-1)\\
    1 & 0  & (-1)^{\tilde{k}_2-4}\frac{1}{2}(\tilde{k}_2-2)(\tilde{k}_2-1) \\ \end{array} }\right).\]

In general, $M$ is an $(\tilde{L}+1)\times (\tilde{L}+1)$ matrix. According with \eqref{eq:normalisedMExpRearranged}, we must be sensitive to whether $\tilde{\mathcal{K}}-m$  contains negative elements distinct from $-m$ (as these do not appear in the sum). To account for these, it is useful to express \begin{equation}\label{eq:binomAsProduct}a_{m}(\tilde{k}_b-m)=\binom{\tilde{k}_b-1}{m-1}=\frac{1}{(m-1)!}\prod_{i=1}^{m-1}(\tilde{k}_b-i).\end{equation}
For $b\geq 1$, the expression on the right hand side of \eqref{eq:binomAsProduct} is equal to $0$ when $\tilde{k}_b<m$. If $b\geq 1$ and $\tilde{k}_b-(\tilde{L}+1)< 0$, then, for $m>\tilde{k}_b$, $\frac{1}{(m-1)!}\prod_{i=1}^{m-1}(\tilde{k}_b-i)=0$. Observe also that $1=\frac{(-1)^{-m-1}}{(m-1)!}\prod_{i=1}^{m-1}(0-i)$. Thus, for $1\leq m \leq \tilde{L}+1$ and $1\leq b \leq \tilde{L}+1$ the matrix entry $M(m,b)$ is given by
\begin{equation}
 M(m,b)=\frac{(-1)^{\tilde{k}_{b-1}-m-1}}{(m-1)!}\prod_{i=1}^{m-1}(\tilde{k}_{b-1}-i).
\end{equation}

We claim that $M$ is invertible. Let us suppose, for now, that the claim holds and see how the result follows. Using the fact that $M$ is invertible, we find from \eqref{eq:normalisedMExpVec} that
\[|\vec{\varepsilon}|\sim h^{D-(\tilde{L}+1)}.\]
Throughout, we have tracked error terms so that $|\vec{\varepsilon}|\ll h^{D-(\tilde{L}+1)}$. As such, we have a contradiction. Therefore, $B(w_1,\epsilon |w_1|)$ can  contain \emph{at most} $\tilde{L}$ roots.

Let us now prove the claim. First, lets multiply the $j$th column of $M$ by $(-1)^{\tilde{k}_j}$ and call the resulting matrix $\widetilde{M}$.  We work to show the column vectors of $\widetilde{M}$ are linearly independent via a Taylor expansion. Observe that the columns of $\widetilde{M}$ are evaluations of the vector polynomial \[\vec{p}(t)=\begin{pmatrix}
     1\\
     -(t-1)\\
     \frac{1}{2!}(t-1)(t-2)\\
      \vdots \\
     (-1)^{\tilde{L}}\frac{1}{\tilde{L}!}(t-1)(t-2)\ldots(t-\tilde{L})
\end{pmatrix}\]
at the points $t=0,\tilde{k}_1$, $\ldots$ $\tilde{k}_{\tilde{L}}$.
Considering the Maclaurin expansion, we can express $p(t)$ as a matrix transformation of the curve,
\[\vec{q}(t)=\begin{pmatrix}
     1\\
     t\\
     \frac{1}{2!}t^2\\
      \vdots \\
    \frac{1}{\tilde{L}!}t^{\tilde{L}}
\end{pmatrix}.\]
Note that, since the $j$th component of $\vec{p}$ is a degree $j-1$ polynomial, we can write \[\vec{p}(t)=T\vec{q}(t),\]
where $T$ is an upper triangular matrix. It is easy to see that, along the diagonal, the entries are non-zero, so that $T$ is invertible. It is well known that $\tilde{L}$ distinct points on the moment curve
\[\vec{r}(t)=\begin{pmatrix}
     t\\
     \frac{1}{2!}t^2\\
      \vdots \\
    \frac{1}{(\tilde{L})!}t^{\tilde{L}}
\end{pmatrix}\]
are in general position. We thus see that 
\[\dim\Span\left\lbrace \vec{q}(0),\vec{q}(\tilde{k}_1),\ldots,\vec{q}(\tilde{k}_{\tilde{L}})\right\rbrace\]
\[=1+\dim\Span\left\lbrace \vec{r}(\tilde{k}_1),\ldots,\vec{r}(\tilde{k}_{\tilde{L}})\right\rbrace\]
\[=\tilde{L}+1.\]
We also know, since $T$ is invertible, that
\[\dim\Span\left\lbrace T\vec{q}(0),T\vec{q}(\tilde{k}_1),\ldots,T\vec{q}(\tilde{k}_{\tilde{L}})\right\rbrace\]
\[=\dim\Span\left\lbrace \vec{q}(0),\vec{q}(\tilde{k}_1),\ldots,\vec{q}(\tilde{k}_{\tilde{L}})\right\rbrace\]
\[=\tilde{L}+1.\]
Therefore, the column vectors of $\widetilde{M}$ are linearly independent, completing the proof that $\widetilde{M}$, and thus $M$, is invertible.
\end{proof}

\section[A series expansion lemma]{A series expansion lemma; proof of Lemma \ref{lem:serDistinguished}}\label{sec:seriesExp}
In this section, we work to prove our main series expansion Lemma \ref{lem:serDistinguished}. The work is of a rather combinatorial flavour. Lemma \ref{lem:serDistinguished} is a corollary of Lemmas \ref{lem:highlightToRecurrence} and \ref{lem:recRelSol}.

\begin{lemma}\label{lem:highlightToRecurrence}For $m$ highlighted roots, $\mathpzc{h}$, including $w_1$ and $D-m$ excluded roots $\mathpzc{e}=\mathcal{T}\backslash\mathpzc{h}$, the symmetric function $S_{D-m}(\mathcal{T})$ can be expressed as \[S_{D-m}(\mathcal{T})=S_{D-m}(\mathpzc{e})+\sum_{j=1}^{D-m}a_m(j,m)(-w_1)^jS_{D-m-j}(\mathcal{T})+\varepsilon_{D-m},\]
where $|\varepsilon_{D-m}|\lesssim \max_{w,w'\in\mathpzc{h}}|w-w'|h^{D-m}$, if $D\geq m\geq 2$, and $\varepsilon_{D-m}=0$, if $m=1$. The coefficients $a_m(j,m)$ are outlined in the following. We first set $b=\max\lbrace 1, m-(D-m)+1\rbrace$. We have that $a_m(\cdot,\cdot)$ is the solution of the below recurrence relation:
\begin{equation}\label{eq:recRel}
\begin{split}
a_m(1,m)=\binom{m}{m-1},a_m(1,m-1)=\binom{m}{m-2},\ldots, a_m(1,b)=\binom{m}{b-1},\\
a_m(1,l)=0\text{, for }l\leq b-1\text{ or }l\geq m+1,\\
a_m(j+1,l)=a_m(j,l-1)\text{, for }l\leq b-1\text{ or }l\geq m+1,\\
a_m(j+1,l)=-\binom{m}{l-1}a_{m}(j,m)+a_m(j,l-1)\text{, for }b\leq l\leq m.
\end{split}
\end{equation}
\end{lemma}

In the case that $m\leq D-m$, i.e. $b=1$, we include a figure representing a step of the recursion for those $1\leq l \leq m$. Figure \ref{fig:recRelStep} represents the dynamics taking us one step from $\left\lbrace a_m(j,m),\ldots, a_m(j,1)\right\rbrace$ to $\left\lbrace a_m(j+1,m),\ldots, a_m(j+1,1)\right\rbrace$, where each arrow represents addition. We suppress the first argument of $a_m$.
\begin{figure}[h]
\includegraphics[width=0.6\textwidth]{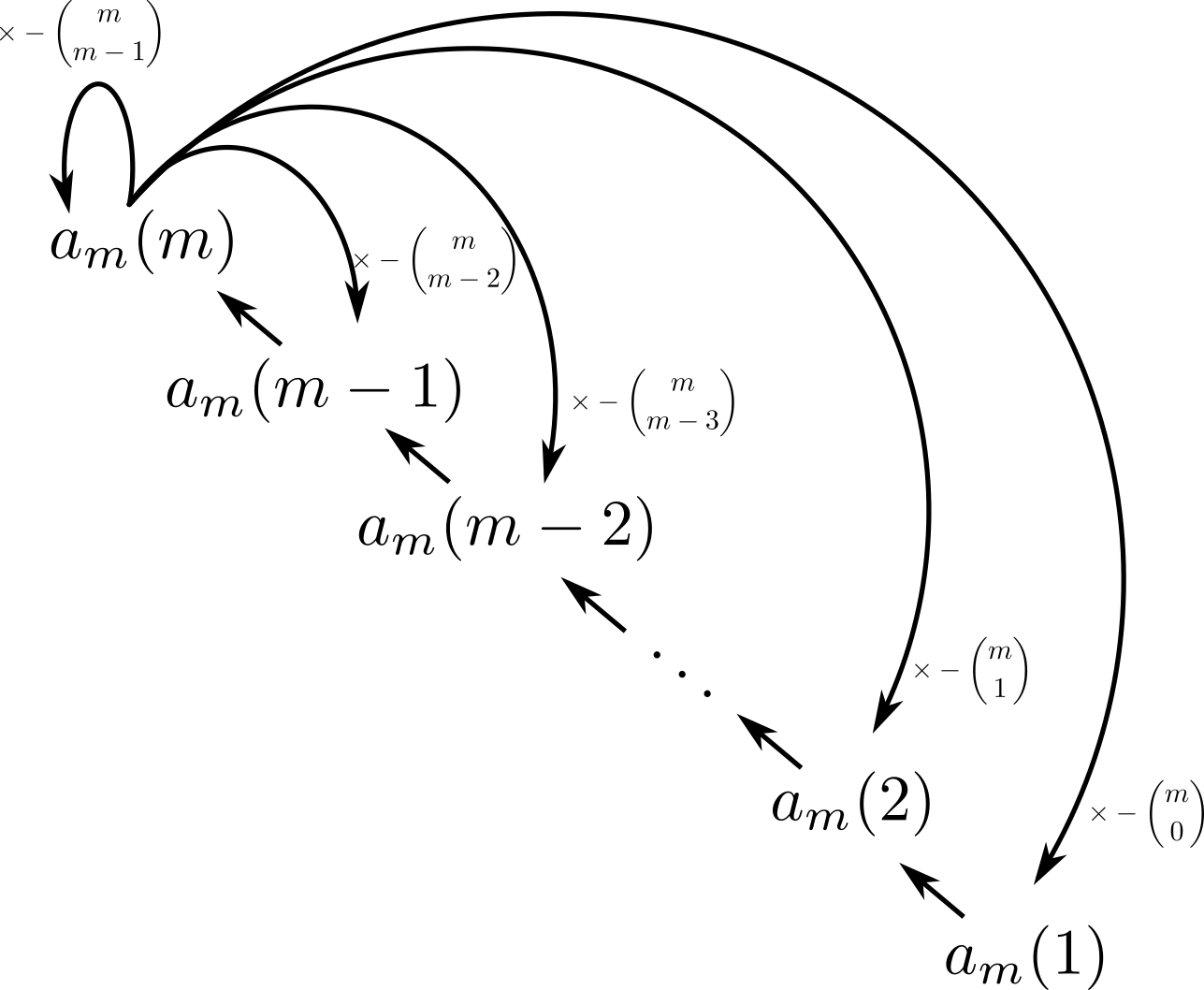}
\centering
\caption{A step of the recursion \eqref{eq:recRel}.}\label{fig:recRelStep}
\end{figure}

\begin{lemma} \label{lem:recRelSol}Let $a_m(\cdot,\cdot)$ satisfy the recurrence relation \eqref{eq:recRel}. For $b\leq l \leq m$ and $1\leq j\leq D-m$, \begin{equation}
a_m(j,l)=(-1)^{j-1}\binom{m-1+j}{l-1}\binom{m-1+j-l}{j-1}.
\end{equation}
In particular, \[a_m(j,m)=(-1)^{j-1}\binom{m-1+j}{m-1},\]
for $1\leq j \leq D-m$.
\end{lemma}
Now we turn to the derivation of our desired series expansion in terms of the above specified recurrence relation.
\begin{proof}[Proof of Lemma \ref{lem:highlightToRecurrence}]
Take $m$ highlighted roots $\mathpzc{h}=\lbrace w_1,w_2,\ldots,w_m\rbrace$. We work from Lemma \ref{lem:serHighlight}, which gives
\[S_{D-m}(\mathcal{T})-S_{D-m}(\mathpzc{e})=\sum_{l=1}^{\min\lbrace D-m, m\rbrace }S_{l}(\mathpzc{h})S_{D-m-l}(\mathpzc{e}).\]

We perform an iterative procedure to obtain the recurrence relation \eqref{eq:recRel}. It is obtained by expressing each of the highest order $S_j(\mathpzc{e})$ in terms of lower order $S_l$ and each $S_{j-l}(\mathpzc{h})$ in terms of a distinguished root $w_1$.  One step of this procedure corresponds to first replacing the instance of $S_j(\mathpzc{e})$ with the largest index $j$ with $S_j(\mathcal{T})-\sum S_i(\mathpzc{h})S_{j-i}(\mathpzc{e})$ and then replacing instances of $S_i(\mathpzc{h})$ with $\binom{m}{i}(-w_1)^{j-l}+\varepsilon_{i}(\mathpzc{h})$, where $|\varepsilon_{i}(\mathpzc{h})|\lesssim \max_{w,w'\in\mathpzc{h}}|w-w'|h^{i-1}.$ 

Let us first consider the simpler case of $m=1$. We see that
\[S_{D-1}(\mathcal{T})-S_{D-1}(\mathpzc{e})=(-w_1)S_{D-2}(\mathpzc{e})=a'_1(1,1)(-w_1)S_{D-2}(\mathpzc{e}).\]
This gives our initialisation of the recurrence relation, with $a'_1(1,1)=1$ with $a'_1(1,l)=0$ for $l\neq 1$. To begin the recurrence, we substitute $S_{D-2}(\mathpzc{e})=S_{D-2}(\mathcal{T})-(-w_1)S_{D-3}(\mathpzc{e})$, which gives \[S_{D-1}(\mathcal{T})-S_{D-1}(\mathpzc{e})=(-w_1)S_{D-2}(\mathcal{T})-(-w_1)^2S_{D-3}(\mathpzc{e})\]
\[=a'_1(2,2)(-w_1)S_{D-2}(\mathcal{T})+a'_1(2,1)(-w_1)^2S_{D-3}(\mathpzc{e}),\]
so that $a'_1(2,2)=a'_1(1,1)=1$ and $a'_1(2,1)=-1=-a'_1(1,1)+a'_1(1,0)$. This continues until we have eliminated all appearances of the $S_l(\mathpzc{h})$ and $S_l(\mathpzc{e})$ for $l>1$, the recurrence relation is seen to be $a'_1(j+1,1)=-a'_1(j,1)=-a'_1(j,1)+a'_1(j,0)$ for $j\leq D-1$ and $a'(j+1,l)=a'(j,l-1)$ for $l\neq 1$.

In the case where $m\geq 2$, it is harder to work with exact expressions, so we will introduce appropriate error terms. We distinguish the root $w_1\in\mathpzc{h}$. Note that, wherever $S_{l}(\mathpzc{h})$ appears, we can write $S_{l}(\mathpzc{h})=\binom{m}{l}(-w_1)^{l}+\varepsilon_{l}(\mathpzc{h})$, where $|\varepsilon_{l}(\mathpzc{h})|\lesssim \max_{w,w'\in\mathpzc{h}}|w-w'|h^{l-1}$, if $m\geq 2$. This is seen by replacing any instance of $w_j$ for $2\leq j \leq m$ with $w_j=w_1+(w_j-w_1)$: the term resulting from the difference on the right hand side of this equation is cast into to the error term $\varepsilon_{l}(\mathpzc{h})$.

For our initialisation, we first modify all terms featuring highlighted roots by introducing an appropriate error and find that
\[S_{D-m}(\mathcal{T})-S_{D-m}(\mathpzc{e})=\sum_{l=1}^{\min\lbrace m, D-m\rbrace}S_{l}(\mathpzc{h})S_{D-m-l}(\mathpzc{e})\]
\[=\sum_{l=1}^{\min\lbrace m, D-m\rbrace}\binom{m}{l}(-w_1)^lS_{D-m-l}(\mathpzc{e})+\left(\sum_{l=1}^{\min\lbrace m, D-m\rbrace}\varepsilon_{l}(\mathpzc{h})S_{D-m-l}(\mathpzc{e})\right)\]
\begin{equation}\label{eq:initRecursmgeq2}=\sum_{l\geq1}a'_{m}(1,m+1-l)(-w_1)^lS_{D-m-l}(\mathpzc{e})+\varepsilon_{D-m}^{(0)},\end{equation}
 where $a'_{m}(1,l)=\binom{m}{l}$ for $1\leq l\leq \min\lbrace m, D-m\rbrace$ and $a'_{m}(1,l)=0$ for $l>\min\lbrace m, D-m\rbrace$ or $l\leq 0$. Here $\varepsilon_{D-m}^{(0)}\coloneqq \sum_{l=1}^{\min\lbrace m, D-m\rbrace}\varepsilon_{l}(\mathpzc{h})S_{D-m-l}(\mathpzc{e})$ is easily verified to satisfy the required error bound $|\varepsilon_{D-m}^{(0)}|\lesssim \max_{w,w'\in\mathpzc{h}}|w-w'|h^{D-m-1}$. Note that, although $S_j(\mathpzc{e})$ appearing in the sum \eqref{eq:initRecursmgeq2} is not defined for $j>D-m$ or $j <0$, the corresponding terms should not be considered as part of the sum since the corresponding coefficients are $0$. We write the initialisation as an infinite series expansion in this way because it makes the expression of our recursion more convenient.

As the first step of the recursion, we will use the equation \[S_{D-m-1}(\mathpzc{e})=S_{D-m-1}(\mathcal{T})-\sum_{j=1}^{\min\lbrace m, D-m-1\rbrace}S_{j}(\mathpzc{h})S_{D-m-1-j}(\mathpzc{e})\]
 \begin{equation}\label{eq:firstSubRecurs}=S_{D-m-1}(\mathcal{T})-\sum_{j\in \ZZ}\mathbbm{1}_{I_1}(j)\binom{m}{j}(-w_1)^jS_{D-m-1-j}(\mathpzc{e})+\varepsilon_{D-m-1}^{(0,1)}(\mathpzc{h}),\end{equation}
 where $I_1=\left[1,\min\lbrace m, D-m-1\rbrace\right]$ and $\varepsilon^{(0,1)}_{D-m-1}\coloneqq \sum_{j\in \ZZ}\mathbbm{1}_{I_1}(j)\varepsilon_{j}(\mathpzc{h})S_{D-m-1-j}(\mathpzc{e})$ satisfies $|\varepsilon^{(0,1)}_{D-m-1}(\mathpzc{h})|\lesssim\max_{\lbrace w,w'\in \mathpzc{h}\rbrace}|w-w'|h^{D-m-2}$.
 
 To carry out the first step of our recursive procedure, we substitute \eqref{eq:firstSubRecurs} into \eqref{eq:initRecursmgeq2}. We end up with
 \[S_{D-m}(\mathcal{T})-S_{D-m}(\mathpzc{e})\]
\[=a'_{m}(1,m)(-w_1)S_{D-m-1}(\mathcal{T})+\sum_{j\in \ZZ}\mathbbm{1}_{I_1}(j)\binom{m}{j}(-a'_{m}(1,m))(-w_1)^{j+1}S_{D-m-1-j}(\mathpzc{e})\]
\[+a'_{m}(1,m)(-w_1)\varepsilon^{(0,1)}_{D-m-1}(\mathpzc{h})\]\[+\sum_{l\geq 2}a'_{m}(1,m+1-l)(-w_1)^lS_{D-m-l}(\mathpzc{e})+\varepsilon^{(0)}_{D-m}(\mathpzc{h})\]
\[=a'_{m}(1,m)(-w_1)S_{D-m-1}(\mathcal{T})+\sum_{j\in \ZZ}\mathbbm{1}_{I_1}(j)\binom{m}{j}(-a'_{m}(1,m))(-w_1)^{j+1}S_{D-m-1-j}(\mathpzc{e})\]\[+\sum_{j\geq 1}a'_{m}(1,m-j)(-w_1)^{j+1}S_{D-m-1-j}(\mathpzc{e})+\varepsilon_{D-m}^{(1)}(\mathpzc{h})\]
\[=\sum_{l\leq 0}a'_{m}(2,m+1-l)(-w_1)^{l+1}S_{D-m-1-l}(\mathcal{T})\]
\[+\sum_{l\geq 1}a'_{m}(2,m+1-l)(-w_1)^{l+1}S_{D-m-1-l}(\mathpzc{e})+\varepsilon_{D-m}^{(1)}(\mathpzc{h})\]
where $|\varepsilon_{D-m}^{(1)}(\mathpzc{h})|\ll h^{D-m}$ and
\begin{equation} a'_{m}(2,m+1-l)=-\mathbbm{1}_{I_1}(l)\binom{m}{l}a'_m(1,m)+a'_{m}(1,m-l).\end{equation}
The procedure continues. The recurrence relation we get is
\begin{equation}\label{eq:derivedRecRel} a'_{m}(j+1,m+1-l)=-\mathbbm{1}_{I_j}(l)\binom{m}{l}a'_m(j,m)+a'_{m}(j,m-l),\end{equation}
where
\begin{equation}\label{eq:intervalIjDef} I_j\coloneqq [1,\min \left\lbrace m,D-m-j\right\rbrace],\end{equation}
where, if $\min \left\lbrace m,D-m-j\right\rbrace<1$, \eqref{eq:intervalIjDef} should be understood as a void interval, i.e. $I_j=\emptyset$. The indicator function $\mathbbm{1}_{I_j}$ appearing in \eqref{eq:derivedRecRel} simply ensures that we do not pick up any symmetric functions below $S_0(\mathcal{T})$ in our series expansion. 

Looking at \eqref{eq:derivedRecRel} and \eqref{eq:intervalIjDef}, we can see that, for $j\geq D-m$, $a'_m(j+1,l)=a'_m(j,l)$ for all $l$. This reflects the fact that we can only carry out the iteration procedure, where we replace instances of $S_{D-m-j}(\mathpzc{e})$ with $S_{D-m-j}(\mathcal{T})$, an error, and terms involving $S_{D-m-j'}(\mathpzc{e})$ for $j'>j$, $D-m-1$ times. The derived series expansion for $S_{D-m}(\mathcal{T})-S_{D-m}(\mathpzc{e})$ is 
\[S_{D-m}(\mathcal{T})-S_{D-m}(\mathpzc{e})\]
\[=\sum_{l\leq 0}a'_{m}(2,m+1-l)(-w_1)^{l+1}S_{D-m-1-l}(\mathcal{T})+\sum_{l\geq 1}a'_{m}(2,m+1-l)(-w_1)^{l+1}S_{D-m-1-l}(\mathpzc{e})\]
\[+\varepsilon_{D-m}^{(1)}(\mathpzc{h})\]
\[=\sum_{l\leq 0}a'_{m}(3,m+1-l)(-w_1)^{l+2}S_{D-m-2-l}(\mathcal{T})+\sum_{l\geq 1}a'_{m}(3,m+1-l)(-w_1)^{l+2}S_{D-m-2-l}(\mathpzc{e})\]
\[+\varepsilon_{D-m}^{(2)}(\mathpzc{h})\]
\[=\ldots\]
\[=\sum_{l\leq 0}a'_{m}(j'+1,m+1-l)(-w_1)^{l+j'}S_{D-m-j'-l}(\mathcal{T})+\sum_{l\geq 1}a'_{m}(j'+1,m+1-l)(-w_1)^{l+j'}S_{D-m-j'-l}(\mathpzc{e})\]
\[+\varepsilon_{D-m}^{(j')}(\mathpzc{h}).\]
Taking $j'= D-m-1$, we find that
\[S_{D-m}(\mathcal{T})-S_{D-m}(\mathpzc{e})=\sum_{l\leq 0}a'_{m}(D-m,m+1-l)(-w_1)^{l+D-m-1}S_{1-l}(\mathcal{T})\]\[+\sum_{l\geq 1}a'_{m}(D-m,m+1-l)(-w_1)^{l+D-m-1}S_{1-l}(\mathpzc{e})+\varepsilon_{D-m}^{(D-m-1)}(\mathpzc{h})\]
\[=\sum_{l\leq 0}a'_{m}(D-m,m+1-l)(-w_1)^{l+D-m-1}S_{1-l}(\mathcal{T})+a'_{m}(D-m,m)(-w_1)^{D-m}S_{0}(\mathpzc{e})+\varepsilon_{D-m}^{(D-m-1)}(\mathpzc{h})\]
\[=\sum_{m-D+2\leq l\leq 1}a'_{m}(D-m+l-1,m)(-w_1)^{l+D-m-1}S_{1-l}(\mathcal{T})+\varepsilon_{D-m}^{(D-m-1)}(\mathpzc{h}).\]
Now observe that, for $l\leq 0$, 
$a'_{m}(j+1,m+1-l)=a'_{m}(j,m-l)$, by definition \eqref{eq:derivedRecRel}. Therefore, for $l\leq 0$, 
$a'_{m}(j',m+1-l)=a'_{m}(j'+l-1,m).$ 
In particular, we have 
\[S_{D-m}(\mathcal{T})-S_{D-m}(\mathpzc{e})\]
\[=\sum_{1\leq j\leq D-m}a'_{m}(j,m)(-w_1)^{j}S_{D-m-j}(\mathcal{T})+\varepsilon_{D-m}^{(D-m-1)}(\mathpzc{h}).\]

 One can easily see that the recurrence relation \eqref{eq:derivedRecRel} differs from the recurrence relation \eqref{eq:recRel}. However, the terms that appear in the series expansion are $a'_m(j,m)$ for $1\leq j \leq D-m$ and it can be verified that these coincide with the same $a_m(j,m)$. \end{proof}

\begin{proof}[Proof of Lemma \ref{lem:recRelSol}]
We want to show that, for $b\leq l \leq m$ and $1\leq j \leq D-m$,
\begin{equation}\label{eq:amjlExplicit}a_m(j,l)=(-1)^{j-1}\binom{m-1+j}{l-1}\binom{m-1+j-l}{j-1},\end{equation}
where $a_m$ satisfies the recurrence relation \eqref{eq:recRel}.
This is true for $j=1$ by definition. We now work by induction. We must show that the right hand side of \eqref{eq:amjlExplicit} satisfies the recurrence relation \eqref{eq:recRel}. We apply the relevant forward expression to see that, for $b\leq l \leq m$ and $j\geq 1$,
\[-\binom{m}{l-1}(-1)^{j-1}\binom{m-1+j}{m-1}+(-1)^{j-1}\binom{m-1+j}{l-2}\binom{m+j-l}{j-1}\]
\[=(-1)^{j}\left(\frac{m!}{(l-1)!(m-l+1)!}\cdot\frac{(m-1+j)!}{j!(m-1)!}-\frac{(m+j-1)!}{(m+j-l+1)!(l-2)!}\cdot\frac{(m+j-l)!}{(j-1)!(m-l+1)!}\right)\]
\[=(-1)^{j}\left(\frac{(m-1+j)!}{(l-1)!(m+j-l+1)!}\right)\left(\frac{m(m+j-l+1)!}{(m-l+1)!j!}-\frac{(m+j-l)!(l-1)}{(j-1)!(m-l+1)!}\right)\]
\[=(-1)^{j}\left(\frac{(m-1+j)!}{(l-1)!(m+j-l+1)!}\right)\left(\frac{(m+j-l)!}{j!(m-l+1)!}\right)\left(m(m+j-l+1)-(l-1)j\right)\]
\[=(-1)^{j}\left(\frac{(m-1+j)!}{(l-1)!(m+j-l+1)!}\right)\left(\frac{(m+j-l)!}{j!(m-l+1)!}\right)\left((m+j)(m+1-l)\right)\]
\[=(-1)^{j}\binom{m+j}{l-1}\binom{m+j-l}{j},\]
as required.
\end{proof}

\part{Explicit root structure}\label{part:ExpRootStruc}
In Section \ref{sec:strata}, we introduced certain reference heights. These reference heights provided an essential scaffold for our root structure analysis. Nevertheless, the tools developed so far tell us nothing explicit about even the size of the reference heights. The reference heights were given in terms of the size of certain roots and thus depended implicitly on the coefficients of our polynomial $\Psi$. In this section, we establish an explicit toolkit for estimating reference heights, from which we are then able to obtain refined structural statements.

We also consider the  rough factorisation of monic $\Psi$ into polynomials corresponding with each tier. We provide an explicit factorised polynomial $\widetilde{\Psi}(t)=\prod_{r=1}^s\widetilde{\Psi}_r(t)$. Our analysis results in a further tool for root finding, in that a suitable small neighbourhood of the roots of the polynomial factor $\widetilde{\Psi}_l$ will contain the roots of $\Psi$ in the tier $\mathcal{T}_l$.
\section{Root tier stratification}\label{sec:heightEst}
We have defined the tier structure in terms of reference heights about which we have no a priori information. The following lemma outlines a useful procedure for estimating the reference heights in an \emph{unknown} regime. 

In this section, we prove our explicit tool for estimating the heights of roots, Lemma \ref{lem:heightEstProcedure}. We also prove Theorem \ref{thm:tierStrucFine}. This is a direct corollary of our later Proposition \ref{prop:tierRefined} considered with reference to Theorem \ref{thm:singleTierStruc}.

\begin{lemma}\label{lem:heightEstProcedure}We fix a set of exponents $k_1<k_2<\ldots <k_L$ and consider polynomials \[\Psi(t)=x+y_1t^{k_1}+\ldots+y_Lt^{k_L},\]
with $y_L\neq 0$. The roots of $\Psi$ are structured into tiers, $\mathcal{T}_1$, $\mathcal{T}_2$, $\ldots$ $\mathcal{T}_s$,  according with Theorem \ref{thm:tierStruc}.
There exists an algorithm outputting a sequence of height estimates
 $\eta_1\geq\eta_2\geq\ldots\geq\eta_a$, where $1\leq s\leq a\leq L$, which satisfy the following.
 
 Associated with $\eta_1$, $\eta_2$,$\ldots$ $\eta_a$ are indices $0=\alpha(0),\alpha(1),\alpha(2),\ldots, \alpha(a)$ such that \[\eta_j^{D(\alpha(j))-D(\alpha(j-1))}=\left|\frac{y_{L-\alpha(j)}}{y_{L-\alpha(j-1)}}\right|.\]
 
For each $0\leq b<a$, we proceed by setting \[\eta_{b+1}=\max_{\alpha(b)<j\leq L}\left|\frac{y_{L-j}}{y_{L-\alpha(b)}}\right|^{\frac{1}{D(j)-D(\alpha(b))}}=\left|\frac{y_{L-\alpha(b+1)}}{y_{L-\alpha(b)}}\right|^{\frac{1}{D(\alpha(b+1))-D(\alpha(b))}}.\]
 
 There exist $0=\beta(0)$, $\beta(1)$, $\beta(2)$, $\ldots$ $\beta(s)= a$ such that
\[\eta_j\sim h(\mathcal{T}_r),\text{ for }\beta(r-1)<j \leq \beta(r).\]
Furthermore, $\alpha(\beta(r))=L(r)$, with the $l_j$ and $L(j)$ defined as in Definition \ref{def:tierDef}.

We can reformulate the previous comparison as follows: for $\alpha(j)$ with $L(r-1)+1\leq \alpha(j)\leq L(r)$, 
\begin{equation}\label{eq:heightEstCompRefHeight}\eta_j\sim h(\mathcal{T}_r).\end{equation}
\end{lemma}

\begin{proof}
Suppose we are in the regime indexed by $(l_1,l_2,\ldots,l_s)$, as in Definition \ref{def:tierDef}. We can easily verify that \begin{equation}\label{eq:heightCompMainTerm}\left|S_{D(L(i))}(\mathcal{R})\right|\sim h(\mathcal{T}_1)^{D(\mathcal{T}_1)}\ldots h(\mathcal{T}_i)^{D(\mathcal{T}_i)}.\end{equation}
Indeed, $(-z_1)(-z_2)\ldots (-z_{D(L(i))})$ is the term of largest magnitude appearing in the sum $S_{D(L(i))}(\mathcal{R})$ and it is comparable in magnitude to $h(\mathcal{T}_1)^{D(\mathcal{T}_1)}\ldots h(\mathcal{T}_i)^{D(\mathcal{T}_i)}.$ Since they must include a root from a smaller tier, all the remaining terms in the sum $S_{D(L(i))}(\mathcal{R})$ are bounded in magnitude by \[h(\mathcal{T}_1)^{D(\mathcal{T}_1)}\ldots h(\mathcal{T}_i)^{D(\mathcal{T}_i)-1} h(\mathcal{T}_{i+1})\ll h(\mathcal{T}_1)^{D(\mathcal{T}_1)}\ldots h(\mathcal{T}_i)^{D(\mathcal{T}_i)}.\]
Similarly, note that, for $L(t-1)+1\leq j\leq L(t)$, \begin{equation}\label{eq:heightCompIntermediateTerm}|S_{D(j)}(\mathcal{R})|\lesssim h(\mathcal{T}_{1})^{D(\mathcal{T}_{1})}h(\mathcal{T}_{2})^{D(\mathcal{T}_{2})}\ldots h(\mathcal{T}_{t-1})^{D(\mathcal{T}_{t-1})}h(\mathcal{T}_{t})^{D(j)-D(L(t-1))}.\end{equation}

Set \[\eta_1=\max_{1\leq j\leq L}|S_{D(j)}(\mathcal{R})|^\frac{1}{D(j)}=|S_{D(\alpha(1))}(\mathcal{R})|^\frac{1}{D(\alpha(1))}=\left|\frac{y_{L-\alpha(1)}}{y_{L}}\right|^{\frac{1}{D(\alpha(1))}}.\]
Since all roots are bounded in magnitude by $h(\mathcal{T}_1)$ and, from \eqref{eq:heightCompMainTerm}, $S_{D(L(1))}(\mathcal{R})\sim h(\mathcal{T}_1)^{D(L(1))}$, we then see that $\eta_1\sim h(\mathcal{T}_1)$.

Next, we set \[\eta_2=\max_{\alpha(1)< j\leq L}\left|\frac{S_{D(j)}(\mathcal{R})}{S_{D(\alpha(1))}(\mathcal{R})}\right|^\frac{1}{D(j)-D(\alpha(1))}\]\[=\left|\frac{S_{D(\alpha(2))}(\mathcal{R})}{S_{D(\alpha(1))}(\mathcal{R})}\right|^\frac{1}{D(\alpha(2))-D(\alpha(1))}=\left|\frac{y_{L-\alpha(2)}}{y_{L-\alpha(1)}}\right|^\frac{1}{D(\alpha(2))-D(\alpha(1))}.\]

Having determined $\eta_1$, $\eta_2$, $\ldots$, $\eta_{i}$ and corresponding $\alpha(1)$, $\alpha(2)$, $\ldots$, $\alpha(i)<L$, we set \begin{equation}\label{eq:heightEstRec}\eta_{i+1}=\max_{\alpha(i)< j\leq L}\left|\frac{S_{D(j)}(\mathcal{R})}{S_{D(\alpha(i))}(\mathcal{R})}\right|^\frac{1}{D(j)-D(\alpha(i))}\end{equation}\[=\left|\frac{S_{D(\alpha(i+1))}(\mathcal{R})}{S_{D(\alpha(i))}(\mathcal{R})}\right|^\frac{1}{D(\alpha(i+1))-D(\alpha(i))}=\left|\frac{y_{L-\alpha(i+1)}}{y_{L-\alpha(i)}}\right|^\frac{1}{D(\alpha(i+1))-D(\alpha(i))}.\]

We wish to ensure that at least one height estimate corresponds to every reference height. This is ensured provided $h_{l(r)}\gg h_{l(r)+1}$ for each $1\leq r \leq s$ with suitable constants in Definition \ref{def:tierDef}. In the procedure defined above, we show that each $y_{L-L(t)}$ appears in one of the expressions for $\eta_j$, \eqref{eq:heightEstRec}. Furthermore, for terms picked up in the procedure, the bound \eqref{eq:heightCompIntermediateTerm} can be upgraded to a comparison. We show this inductively.

First, take the largest $0\leq \alpha(\beta_1-1)$ such that $\alpha(\beta_1-1)<L(1)$. Observe that \[\left|S_{D(L(1))}(\mathcal{R})\right|^{\frac{1}{D(L(1))}}=\left|\frac{y_{L-L(1)}}{y_L}\right|^{\frac{1}{D(L(1))}}\sim h(\mathcal{T}_1),\] by \eqref{eq:heightCompMainTerm}. Thus we can see, by definition of our height estimates, that, if $\alpha(\beta_1-1)>0$, then $\left|S_{D(\alpha(\beta_1-1))}\right|^{ \frac{1}{ D(\alpha(\beta_1-1)) } }\gtrsim h(\mathcal{T}_1)$. Considering \eqref{eq:heightCompIntermediateTerm}, we then see that either $\alpha(\beta_1-1)=0$ or \begin{equation}\label{eq:SalphaInitialComp}\left|S_{D(\alpha(\beta_1-1))}\right|^{ \frac{1}{ D(\alpha(\beta_1-1)) } }\sim h(\mathcal{T}_1).\end{equation}
We now wish to show that $\alpha(\beta_1)=L(1)$. Suppose for contradiction that $\alpha(\beta_1)>L(1)$. Set $\tilde{h}=h(\mathcal{T}_{2})$ and $h=h(\mathcal{T}_{1})$. We know that $\tilde{h}\ll h$. Using the equations \eqref{eq:heightCompIntermediateTerm} and \eqref{eq:SalphaInitialComp} to bound the size of the symmetric functions in terms of height estimates, we see that
\[\eta_{\beta_1}=\left|\frac{y_{L-\alpha(\beta_1)}}{y_{L-\alpha(\beta_1-1)}}\right|^{\frac{1}{D(\alpha(\beta_1))-D(\alpha(\beta_1-1))}}\]
\[=\left|\frac{S_{D(\alpha(\beta_1))}(\mathcal{R})}{S_{D(\alpha(\beta_1-1))}(\mathcal{R})}\right|^{\frac{1}{D(\alpha(\beta_1))-D(\alpha(\beta_1-1))}}\]
\[\lesssim\left|\frac{\tilde{h}^{D(\alpha(\beta_1))-D(L(1))}h^{D(L(1))}}{h^{D(\alpha(\beta_1-1))}}\right|^{\frac{1}{D(\alpha(\beta_1))-D(\alpha(\beta_1-1))}}\]
\[\leq\left|\frac{\tilde{h}^{D(\alpha(\beta_1))-D(L(1))}}{h^{D(\alpha(\beta_1))-D(L(1))}}h^{D(\alpha(\beta_1)))-D(\alpha(\beta_1-1))}\right|^{\frac{1}{D(\alpha(\beta_1))-D(\alpha(\beta_1-1))}}\]
\[\ll h=\left|h^{D(L(1))-D(\alpha(\beta_1-1))}\right|^{\frac{1}{D(L(1))-D(\alpha(\beta_1-1))}}\]
\[\sim\left|\frac{S_{D(L(1))}(\mathcal{R})}{S_{D(\alpha(\beta_1-1))}(\mathcal{R})}\right|^{\frac{1}{D(L(1))-D(\alpha(\beta_1-1))}}\]
\[=\left|\frac{y_{L-L(1)}}{y_{L-\alpha(\beta_1-1)}}\right|^{\frac{1}{D(L(1))-D(\alpha(\beta_1-1))}}.\]
This contradicts the definition of $\eta_{\beta_1}$, so we must have that $\alpha(\beta_1)=L(1)$. Furthermore, for $0=\beta(0)< i \leq \beta(1)=\beta_1$, \[\eta_{i}\sim h(\mathcal{T}_1).\]

We proceed inductively. Fix some index $r$. Suppose that $\beta(r-1)$ is such that $\alpha(\beta(\tilde{r}))=L(\tilde{r})$ for $1\leq \tilde{r}\leq r-1$ and, for $1\leq \tilde{r}<r$  and $\beta(\tilde{r}-1)<i\leq \beta(\tilde{r})$, \begin{equation}\label{eq:heightEstComparisonA}\eta_i\sim h(\mathcal{T}_{\tilde{r}}).\end{equation}
We then choose $\beta_{r}$ maximally so that $\alpha(\beta(r-1))=L(r-1)\leq \alpha(\beta_{r}-1)<L(r)$. Similarly to our proof of \eqref{eq:SalphaInitialComp}, it is routine to verify that \[\left|\frac{y_{L(r)}}{y_{L(r-1)}}\right|^{\frac{1}{D(L(r))-D(\alpha(\beta(r-1)))}}=\left|\frac{S_{D(L(r))}(\mathcal{R})}{S_{D(\alpha(L(r-1))}(\mathcal{R})}\right|^{\frac{1}{D(L(r))-D(\alpha(\beta(r-1)))}}\sim h(\mathcal{T}_{r}),\]
we easily see that, for $\beta(r-1)<i\leq \beta(r)$, \begin{equation}\label{eq:heightEstComparisonB}\eta_{i}\sim h(\mathcal{T}_{r}).\end{equation} As a consequence of \eqref{eq:heightEstComparisonA} and \eqref{eq:heightEstComparisonB}, we find that, for $\beta(r-1)<i\leq \beta(r)$,
\[|S_{D(\alpha(i))}(\mathcal{R})|=\left|\prod_{l=1}^i\eta_i^{D(\alpha(l))-D(\alpha(l-1))}\right|\] 
\begin{equation}\label{eq:heightCompDerivedTerm}\sim h(\mathcal{T}_{1})^{D(\mathcal{T}_{1})}h(\mathcal{T}_{2})^{D(\mathcal{T}_{2})}\ldots h(\mathcal{T}_{r-1})^{D(\mathcal{T}_{r-1})}h(\mathcal{T}_{r})^{D(\alpha(i))-D(L(r-1))}.\end{equation} As previously, we now wish to show that $\alpha(\beta_{r})=L(r)$. Suppose, then, to find a contradiction, that $\alpha(\beta_{r})>L(r)$. Similarly to \eqref{eq:heightCompDerivedTerm}, we find, setting $h=h(\mathcal{T}_r)$ and $\tilde{h}=h(\mathcal{T}_{r+1})$, that
\[\eta_{\beta_{r}}=\left|\frac{y_{L-\alpha(\beta_{r})}}{y_{L-\alpha(\beta_{r}-1)}}\right|^{\frac{1}{D(\alpha(\beta_{r}))-D(\alpha(\beta_{r}-1))}}=\left|\frac{S_{D(\alpha(\beta_{r}))}(\mathcal{R})}{S_{D(\alpha(\beta_{r}-1))}(\mathcal{R})}\right|^{\frac{1}{D(\alpha(\beta_{r}))-D(\alpha(\beta_{r}-1))}}\]
\[\lesssim\left|\frac{h(\mathcal{T}_{1})^{D(\mathcal{T}_{1})}h(\mathcal{T}_{2})^{D(\mathcal{T}_{2})}\ldots h(\mathcal{T}_{r})^{D(\mathcal{T}_{r})}h(\mathcal{T}_{r+1})^{D(\alpha(\beta_{r}))-D(L(r))}}{h(\mathcal{T}_{1})^{D(\mathcal{T}_{1})}h(\mathcal{T}_{2})^{D(\mathcal{T}_{2})}\ldots h(\mathcal{T}_{r-1})^{D(\mathcal{T}_{r-1})}h(\mathcal{T}_{r})^{D(\alpha(\beta_{r}-1))-D(L(r-1))}}\right|^{\frac{1}{D(\alpha(\beta_{r}))-D(\alpha(\beta_{r}-1))}}\]
\[\sim\left|\frac{h^{D(L(r))-D(L(r-1))}\tilde{h}^{D(\alpha(\beta_{r}))-D(L(r))}}{h^{D(\alpha(\beta_{r}-1))-D(L(r-1))}}\right|^{\frac{1}{D(\alpha(\beta_{r}))-D(\alpha(\beta_{r}-1))}}\]
\[=\left|h^{D(\alpha(\beta_{r}))-D(\alpha(\beta_{r}-1))}\frac{\tilde{h}^{D(\alpha(\beta_{r}))-D(L(r))}}{h^{D(\alpha(\beta_{r}))-D(L(r))}}\right|^{\frac{1}{D(\alpha(\beta_{r}))-D(\alpha(\beta_{r}-1))}}\]
\[\ll h\sim \left|\frac{y_{L(r)}}{y_{L(r-1)}}\right|^{\frac{1}{D(L(r))-D(\alpha(\beta(r-1)))}},\]
which contradicts the definition of $\eta_{\beta_r}$.

\end{proof}

We can use the procedure from Lemma \ref{lem:heightEstProcedure} to obtain the refined structural result, Theorem \ref{thm:tierStrucFine}. Here, with additional restrictions on the coefficients, we have stronger control on the root structure as a consequence of the explicit estimates for the reference heights. The following proposition feeds directly into our result on the structure of roots within a given tier, Theorem \ref{thm:singleTierStruc}, to give the refined structural result, Theorem \ref{thm:tierStrucFine}, as a corollary. 

\begin{proposition}\label{prop:tierRefined}
Fix the set of exponents $k_1<k_2<\ldots<k_L$. We consider polynomials \[\Psi(t)=x+y_1t^{k_1}+\ldots+y_Lt^{k_L}\]
such that \[\max_{1\leq j\leq L}|y_j|^\frac{1}{k_j}\leq 1.\]
Take some $\gamma\in (0,1]$ and suppose, additionally, that
\[|y_m|^\frac{1}{k_m}\geq \gamma ,\]
and, for $n>m$,
\[|y_n|^\frac{1}{k_n}\leq \delta,\]
for some $m$ and some suitably small $\delta=\delta(\gamma)>0$. Then the roots of $\Psi$ can be classified into large and small tiers $\mathcal{T}_1$, $\ldots$ $\mathcal{T}_{s(1)}$ and $\mathcal{T}_{s(1)+1}$, $\ldots$ $\mathcal{T}_{s(1)+s(2)}$ which satisfy the following.

First, we have that $0\leq s(1)\leq s(1)+s(2)=s\leq L$ and, additionally $s(1)\leq L-m+1$. If $s(2)\geq 1$, then $s(1)\leq L-m$. 

We refer to those tiers $\mathcal{T}_r$ with $1\leq r\leq s(1)$ as the large tiers. For any root $w$ in a large tier, we have that $|w|\gtrsim_{\gamma}1$. In the case that $s(2)\geq 1$, we refer to those tiers $\mathcal{T}_r$ with $s(1)+1\leq r\leq s(1)+s(2)$ as the small tiers. For any root $w$ in a small tier, we have that $|w|\lesssim_{\gamma}1$. 

The tiers are well separated: for any choice of $w_j\in \mathcal{T}_j$, we have that \[|w_1|\ll |w_2|\ll\ldots\ll|w_{s}|.\] 

Finally, we have the following. In the case that $s(2)\geq 1$, we have that $L(s(1))=L-m$. If $s(2)=0$, then, $l_{s(1)}\geq m$ and, for $L-m<L(s-1)+j<L$,
\[\left|S_{D(L(s-1)+j)}(\mathcal{R})\right|\ll h(\mathcal{T}_1)^{D(\mathcal{T}_1)}h(\mathcal{T}_2)^{D(\mathcal{T}_2)}\ldots h(\mathcal{T}_s)^{D(L(s-1)+j)-D(L(s-1))}.\]
\end{proposition}

Before proceeding with the proof of Proposition \ref{prop:tierRefined}, let us show how Theorem \ref{thm:tierStrucFine} is obtained as a corollary.
\begin{proof}[Proof of Theorem \ref{thm:tierStrucFine}]
Proposition \ref{prop:tierRefined} already gives the large and small tiers, and the required control on their size. It remains to determine the way in which roots may cluster. To do so, we work with reference to the final paragraph of the proposition's statement.

Let us first consider the case that $s(2)\geq 1$. Here we can see, since $L(s(1))=L-m$, that for $1\leq r \leq s(1)$, $l_r\leq L(s(1))\leq L-m$. Therefore $|\mathcal{D}(\mathcal{T}_r)|-1= l_r\leq L-m$ and, by Theorem \ref{thm:singleTierStruc}, $B(w,\epsilon |w|)$ contains \emph{at most} $L-m$ roots for any root $w$ in a large tier $\mathcal{T}_r$. Likewise, for $s(1)+1\leq r \leq s(1)+s(2)$, $l_r\leq L-L(s(1))=m$ so that, for any root $w$ in a small tier $\mathcal{T}_r$, $B(w,\epsilon|w|)$ contains \emph{at most} $m$ roots.

In the case that $s(2)=0$, we work as follows. We consider a tier $\mathcal{T}_r$. If $l_r\leq L-m+1$, then, as above, one has that for any root $w$ the tier $\mathcal{T}_r$, $B(w,\epsilon|w|)$ contains \emph{at most} $L-m+1$ roots. Otherwise, $L(r)\geq l_r>L-m+1$ so that $L(s)-L(r)=L-L(r)<m-1$. Since $l_s\geq m$ we must then have that $r=s$ because otherwise we would have $m\leq l_s\leq L(s)-L(r)<m-1$. According with Lemma \ref{lem:symEqWithinTier}, for $1\leq j\leq l_s$, we then approximate $S_{D_j(\mathcal{T}_s)}(\mathcal{T}_s)$ by  \[\frac{S_{D(L(s-1)+j)}(\mathcal{R})}{S_{D(L(s-1))}(\mathcal{T}_1\cup\ldots\cup\mathcal{T}_{s-1})}.\] In particular,  we have for $1\leq j\leq l_s$ that
\[\left|S_{D_j(\mathcal{T}_s)}(\mathcal{T}_s)-\frac{S_{D(L(s-1)+j)}(\mathcal{R})}{S_{D(L(s-1))}(\mathcal{T}_1\cup\ldots\cup\mathcal{T}_{s-1})}\right|\ll h(\mathcal{T}_s)^{D_j(\mathcal{T}_s)}.\]
Combining this with Proposition \ref{prop:tierRefined}, for $L-m<L(s-1)+j<L$, we see that 
\[\left|S_{D_j(\mathcal{T}_s)}(\mathcal{T}_s)\right|\ll h(\mathcal{T}_s)^{j}.\]
From this, it is a matter of counting to see that \emph{at most} $L-m+1$ non-trivial symmetric functions $S_j(\mathcal{T}_s)$ are substantial: we can choose $\widetilde{\mathcal{D}}(\mathcal{T}_s)\subset\mathcal{D}(\mathcal{T}_s)$ in Theorem \ref{thm:singleTierStruc}, with $|\widetilde{\mathcal{D}}(\mathcal{T}_s))|-1\leq L-m+1$, so that, for $j\notin \widetilde{\mathcal{D}}(\mathcal{T}_s)$, $\left|S_j(\mathcal{T}_s)\right|\ll h(\mathcal{T}_s)^{j}$. Therefore, by Theorem \ref{thm:singleTierStruc}, for any root $w$, $B(w,\epsilon|w|)$ contains \emph{at most} $L-m+1$ roots.
\end{proof}

\begin{proof}[Proof of Proposition \ref{prop:tierRefined}]
We consider what the supposed conditions tell us under the height estimation procedure, Lemma \ref{lem:heightEstProcedure}. We obtain a sequence of reference heights $\eta_1,\eta_2,\ldots$ and corresponding indices $\alpha(1),\alpha(2),\ldots$ satisfying the conditions of that lemma. 

In the case that $m=L$, there is nothing to prove. The height estimates are all $\lesssim_{\gamma}1$. We set $s(1)=0$ so that $L(s(1))=0$ and each tier is a small tier. In what follows, we consider the case $m<L$.

Firstly, observe that $\left|\frac{y_m}{y_L}\right|^{\frac{1}{D(L-m)}}\geq \left|\frac{\gamma^{k_m}}{\delta^{k_L}}\right|^{\frac{1}{D(L-m)}}\gg 1$, provided $\delta$ is sufficiently small. As such, we are guaranteed to pick up a number of \emph{large} height estimates $\eta_j\gtrsim_{\gamma}1$.

Let $0\leq i_0$ be chosen maximally so that $\alpha(i_0)\leq L-m$. We either have that $\alpha(i_0)=L-m$ or $\alpha(i_0)<L-m$ and we first split our analysis by these cases.  Since, for $n>m$, \[\left|\frac{y_m}{y_n}\right|^{\frac{1}{D(L-m)-D(L-n)}}\gtrsim_{\gamma}1\]
it is easy to see that 
\[\eta_{\max\lbrace i_0,1\rbrace}\gtrsim_{\gamma}1.\]

In the first case, where $L-\alpha(i_0)=m$, note that $i_0\geq 1$. We then observe that, provided we choose $\delta$ is chosen sufficiently small depending on $\gamma$, for $1\leq j<m$,
\[\left|\frac{y_{j}}{y_{m}}\right|^{\frac{1}{D(L-j)-D(L-m)}}\leq \left|\frac{1}{\gamma^{k_m}}\right|^{\frac{1}{D(L-j)-D(L-m)}}\]
\begin{equation}\label{eq:lowerHeightEstsDontAppear}\ll \left|\frac{\gamma^{k_m}}{\delta^{k_{L-\alpha(i_0-1)}}}\right|^{\frac{1}{D(L-m)-D(\alpha(i_0-1))}}\leq\left|\frac{y_m}{y_{L-\alpha(i_0-1)}}\right|^{\frac{1}{D(L-m)-D(\alpha(i_0-1))}}= \eta_{i_0},\end{equation}
 by definition of the height estimates. Thus we see that either $\eta_{i_0+1}\ll \eta_{i_0}$ or, if $\eta_{i_0+1}\gtrsim \eta_{i_0}$, we must have that $\eta_{i_0+1}=\left|\frac{x}{y_m}\right|^{\frac{1}{D(L)-D(L-m)}}$.

We now consider what happens in the height estimation procedure in the case that $\alpha(i_0)<L-m$. 
We observe that, for $1\leq j<m$,
\[\left|\frac{y_{j}}{y_{L-\alpha(i_0)}}\right|^{\frac{1}{D(L-j)-D(\alpha(i_0))}}\leq\left|\frac{1}{y_{L-\alpha(i_0)}}\right|^{\frac{1}{D(L-j)-D(\alpha(i_0))}}\]
\begin{equation}\label{eq:lowerHeightEstsDontAppear2}\ll \left|\frac{y_m}{y_{L-\alpha(i_0)}}\right|^{\frac{1}{D(L-m)-D(\alpha(i_0))}},\end{equation}
where one can verify that the last inequality holds because it is satisfied in the extreme:
\[\left|\frac{1}{\delta^{k_{L-\alpha(i_0)}}}\right|^{\frac{1}{D(L-j)-D(\alpha(i_0))}}\ll \left|\frac{\gamma^{k_j}}{\delta^{k_{L-\alpha(i_0)}}}\right|^{\frac{1}{D(L-m)-D(\alpha(i_0))}}.\]
Thus, in the case that $\alpha(i_0)<L-m$, we necessarily have that  $\eta_{i_0+1}=\left|\frac{x}{y_{L-\alpha(i_0)}}\right|^{\frac{1}{D(L)-D(\alpha(i_0))}}$ and $\alpha(i_0+1)=L$, since we have specified that $y_m$ contributes to no height estimate by the condition $\alpha(i_0)<L-m$.

We continue our analysis by splitting according to the control between height estimates $\eta(i_0)$ and $\eta(i_0+1)$. Firstly, we analyse the situation where $\eta(i_0+1)\ll \eta(i_0)$, which can only occur if $L-m=\alpha(i_0)$. Secondly, we analyse the situation where either $\eta(i_0+1)\gtrsim \eta(i_0)$ or where $i_0=0$, which can occur with $L-m=\alpha(i_0)$ or with $L-m>\alpha(i_0)$. 

The first scenario is where $\eta(i_0+1)\ll \eta(i_0)$. Here, we have that $\alpha(i_0)=L-m$. This is the $s(2)\geq 1$ case. According with Lemma \ref{lem:heightEstProcedure}, we choose $s(1)$ so that the large tiers $\mathcal{T}_r$, for $1\leq r\leq s(1)$ are those corresponding with the height estimates $\eta_1,\eta_2,\ldots,\eta_{i_0}$. We must have, by \eqref{eq:heightEstCompRefHeight} from Lemma \ref{lem:heightEstProcedure}, that $L(s(1))=\alpha(i_0)=L-m$. We can also see that \[\eta_{i_0}=\left|\frac{y_m}{y_{L-\alpha(i_0-1)}}\right|^{\frac{1}{D(L-m)-D(\alpha(i_0-1))}}\gtrsim_{\gamma}1,\]
and, by Lemma \ref{lem:heightEstProcedure}, for roots $w$ in large tiers,
\[|w|\gtrsim_{\gamma}1.\]
There are also small tiers of roots: following Lemma \ref{lem:heightEstProcedure}, these are the tiers corresponding with the height estimates $\eta_{i_0+1},\eta_{i_0+2},\ldots\eta_{a}$. It is easy to see that $1\gtrsim_{\gamma}\eta_{i_0+1}$, since \[\eta_{i_0+1}=\left|\frac{y_{L-\alpha(i_0+1)}}{y_m}\right|^{\frac{1}{D(\alpha(i_0+1))-D(L-m)}}\leq \left|\frac{1}{\gamma}\right|^{\frac{1}{D(\alpha(i_0+1))-D(L-m)}},\]
and, by Lemma \ref{lem:heightEstProcedure}, for roots $w$ in small tiers,
\[|w|\lesssim_{\gamma}1.\]

In the second scenario,  $\eta(i_0)\sim \eta(i_0+1)$ or there is only one height estimate and $\alpha(1)=L$. In either case, $\alpha(i_0+1)=L$, as we previously showed how, in this case, we must have $\eta_{i_0+1}=\left|\frac{x}{y_m}\right|^{\frac{1}{D(L)-D(L-m)}}$. This is the $s(2)=0$ case, where all tiers are large. By Lemma \ref{lem:heightEstProcedure}, $l_s\geq \alpha(i_0+1)-\alpha(i_0)\geq L-(L-m)=m$. In this case, to conclude the proof, we must establish control the size of the symmetric functions $S_{D(L-j)}(\mathcal{R})$ for $1\leq j<m$. For these $j$, note that  $L-j>L-m\geq L-l_s=L(s-1)$.
We work to show that that, for $1\leq j<m$, \[\left|S_{D(L-j)}(\mathcal{R})\right|\ll h(\mathcal{T}_1)^{D(\mathcal{T}_1)}h(\mathcal{T}_2)^{D(\mathcal{T}_2)}\ldots h(\mathcal{T}_s)^{D(L-j)-D(L(s-1))}.\] Because we know that $\eta_{i_0+1}=\left|\frac{x}{y_{L-\alpha(i_0)}}\right|^{\frac{1}{D(L)-D(L-\alpha(i_0))}}$ is the final height estimate, \begin{equation}\label{eq:finalHeightEst}\eta_{i_0+1}\sim h(\mathcal{T}_s),\end{equation}
by Lemma \ref{lem:heightEstProcedure}. We claim that that, for $1\leq j<m$, 
\begin{equation}\label{eq:claimedSymFuncBound}\left|\frac{y_j}{y_{L-\alpha(i_0)}}\right|^{\frac{1}{D(L-j)-D(L-\alpha(i_0))}}\ll h(\mathcal{T}_s).\end{equation}
 Assuming for now that \eqref{eq:claimedSymFuncBound} holds, also using the inequality \eqref{eq:heightEstCompRefHeight} from the statement of Lemma \ref{lem:heightEstProcedure}, we see that, for $1\leq j<m$, \[\left|S_{D(L-j)}(\mathcal{R})\right|=\left|\frac{y_j}{y_L}\right|\]
\[=\left|\frac{y_{\alpha(1)}}{y_{L}}\right|\ldots\left|\frac{y_{L-\alpha(i_0)}}{y_{L-\alpha(i_0-1)}}\right|\left|\frac{y_j}{y_{L-\alpha(i_0)}}\right|\]
\[\ll h(\mathcal{T}_1)^{D(\mathcal{T}_1)}\ldots  h(\mathcal{T}_{s-1})^{D(\mathcal{T}_{s-1})}h(\mathcal{T}_s)^{D(L-j)-D(L(s-1))},\]
which, after reindexing, is the desired error bound. To see this, we consider $1\leq j'<l_s$ such that $L-(L(s-1)+j')=l_s-j'<m$: we set $j=L-(L(s-1)+j')$, which ranges between $1$ and $m-1$, as in the proposition's statement.

To conclude, we prove our claimed inequality \eqref{eq:claimedSymFuncBound}. In the case that $L-\alpha(i_0)=m$, this is a direct consequence of \eqref{eq:lowerHeightEstsDontAppear} upon observing from Lemma \ref{lem:heightEstProcedure} that  $\eta_{i_0}\sim h(\mathcal{T}_s)$. In the case that $\alpha(i_0)<L-m$, we use \eqref{eq:lowerHeightEstsDontAppear2} and Lemma \ref{lem:heightEstProcedure}: if $i_0=0$, then there is one height estimate $\eta_1=\left|\frac{x}{y_L}\right|^{\frac{1}{k_L}}\geq\left|\frac{y_m}{y_L}\right|^{\frac{1}{D(L-m)}}\gg \left|\frac{y_j}{y_{L-\alpha(i_0)}}\right|^{\frac{1}{D(L-j)-D(L-\alpha(i_0))}}$, if $i_0\geq 1$, then $\eta_{i_0+1}\geq \left|\frac{y_m}{y_{L-\alpha(i_0)}}\right|^{\frac{1}{D(L-m)-D(\alpha(i_0))}}\gg \left|\frac{y_j}{y_{L-\alpha(i_0)}}\right|^{\frac{1}{D(L-j)-D(L-\alpha(i_0))}}$. Referring to  \eqref{eq:finalHeightEst}, the inequality follows.
\end{proof}

\begin{remark}
As for Proposition \ref{prop:tierRefined}, Lemma \ref{lem:heightEstProcedure} and the procedure it outlines can be used to obtain other refinements of the main structural result, Theorem \ref{thm:tierStruc}. For example, if we had that \[\left|\frac{y_{L-1}}{y_L}\right|^{\frac{1}{d(1)}}\gg \left|\frac{y_{L-2}}{y_{L-1}}\right|^{\frac{1}{d(2)}}\gg \ldots \gg \left|\frac{x}{y_1}\right|^{\frac{1}{d(L)}}>0,\]
then we would have $L$ tiers of roots which are separated and, for sufficiently small $\epsilon$ and some root $z\in \mathcal{R}$, $B(z,\epsilon|z|)$ contains only the root $z$. 
\end{remark}
\section{Rough factorisation}\label{sec:nearFact}
For notational reasons, we consider monic polynomials in this section:
\[\Psi(t)=\sum_{j=0}^Ly_jt^{k_j},\]
where $k_0=0$, $y_0\neq 0$, and $y_L=1$. In this section, we provide a rough factorisation of monic polynomials with a well separated tier structure. To this end, let us suppose that throughout this section we are in the regime indexed by $(l_1,l_2,\ldots,l_s)$, as in Definition \ref{def:tierDef}. Here there are $s$ tiers, $\mathcal{T}_1$, $\ldots$, $\mathcal{T}_s$, containing $D(\mathcal{T}_1)=|\mathcal{T}_1|$, $\ldots$, $D(\mathcal{T}_s)=|\mathcal{T}_s|$ roots, respectively. The tier regime may be roughly characterised by
\begin{equation}\label{eq:refHeightSep}h(\mathcal{T}_1)\gg h(\mathcal{T}_2)\gg\ldots \gg h(\mathcal{T}_s),\end{equation}
for some suitable choice of constants.

Recall the definition of the $k_j(\mathcal{T}_r)$ exponents for a given tier, Definition \ref{def:distIndicesK}. More explicitly, for $0\leq j\leq l_r$, we can write $k_j(\mathcal{T}_r)=k_{L-L(r)+j}-k_{L-L(r)}$.
Recall also the Definition \ref{def:distIndices} of the distinguished indices $D_j(\mathcal{T}_r)$, which we can write more explicitly as $D_j(\mathcal{T}_r)=D(L(r-1)+j)-D(L(r-1))=k_{L-L(r-1)}-k_{L-L(r-1)-j}$.

\begin{definition}\label{def:approxFactorisation}If we are in the regime given in Definition \ref{def:tierDef} indexed by $(l_1,l_2,\ldots,l_s)$, then, for $1\leq j \leq s$, we define the monic polynomial
\begin{equation}\label{eq:factorComp}\widetilde{\Psi}_j(t)\coloneqq\frac{1}{y_{L-L(j-1)}}\sum_{i=0}^{l_j}y_{L-L(j)+i}t^{k_i(\mathcal{T}_j)}.
\end{equation}
We define the  rough factorisation of $ \Psi(t)$ by
\begin{equation}
    \label{eq:almostFact}\widetilde{\Psi}(t)\coloneqq\prod_{j=1}^{s}\widetilde{\Psi}_j(t).
\end{equation}
We denote the roots of $\widetilde{\Psi}_j$ by $\widetilde{\mathcal{T}_j}$ and the roots of $\widetilde{\Psi}$ by $\widetilde{\mathcal{R}}$.
\end{definition}

Away from the zeros of the polynomial, the  rough factorisation, \eqref{eq:almostFact}, we seek should be quantifiably close to the original $ \Psi$. Furthermore, the root structure of $ \Psi$ and the root structure of its rough factorisation should be closely related. In particular, we have Theorem \ref{thm:approxFact}.

\begin{theorem}\label{thm:approxFact}
There exists a polynomial $\widetilde{\Psi}=\prod_{l=1}^s\widetilde{\Psi}_l(t)$, with $\widetilde{\Psi}_l(t)$ given by \eqref{eq:factorComp}, which roughly factorises $ \Psi$ in the following sense.

The polynomial $\widetilde{\Psi}_l$ has roots, $\widetilde{\mathcal{T}}_l$, which are all of comparable magnitude. Furthermore, for roots $w_j\in \widetilde{\mathcal{T}}_j$, we have that 
\[|w_1|\gg|w_2|\gg\ldots\gg|w_s|.\]

There exists a covering, $N(\mathcal{R})$, of the roots, $\mathcal{R}\subset \CC$, of $ \Psi$ which satisfies the following. Each connected component of $N(\mathcal{R})$, which we call a cell, is given by a ball. Each cell containing contains at most $L$ roots. For a cell $B$ containing exactly $m$ roots of $ \Psi$, $B$ contains exactly $m$ roots of $\widetilde{\Psi}$.

For $t\notin N(\mathcal{R})$,
\[| \Psi(t)-\widetilde{\Psi}(t)|\ll | \Psi(t)|.\] 
\end{theorem}

One of the strongest similarities between our work and that of Hickman and Wright \cite{hickWriLP} is to be found in our proof of the rough factorisation theorem. In particular, both proofs follow a method of contradiction and taking suitable estimating the valuation of the respective polynomials. Hickman and Wright essentially consider the related root structures more general family of close polynomials, while here we consider only the rough factorisation and the original polynomial. 
 
The analysis in this section requires strong separation of the height estimates in the specification of the tier regime, Definition \ref{def:tierDef}, as we will see. It should be noted, however, that the results of Part \ref{part:ImpRootStruc} do not require such strong separation, although this is not something we specify in this paper. All of the results in this section should be understood as valid with respect to a tier regime specified with sufficiently strong separation in \eqref{eq:refHeightSep}. 

To begin with, let us bound the difference of $ \Psi(t)$ and $\widetilde{\Psi}(t)$. 
\begin{lemma}
\label{lem:factErr}With \begin{equation}\label{eq:factErrDef}E(t)\coloneqq \widetilde{\Psi}(t)-\Psi(t),\end{equation}
we have that, for each $1\leq r \leq s$ and $h(\mathcal{T}_{r+1})\ll|t|\lesssim h(\mathcal{T}_r)$, that
\begin{equation}\label{eq:factErrBnd}\left|E(t)\right|\leq \epsilon_f \left(\prod_{i=1}^r h(\mathcal{T}_i)^{D(\mathcal{T}_{i})}\right)|t|^{\sum_{j=r+1}^{s}D(\mathcal{T}_{j})},\end{equation}
including for $r=s$ and $r=1$, subject to the understanding that $h(\mathcal{T}_0)=\infty$ and $h(\mathcal{T}_{s+1})=0$. Here the constant $\epsilon_f$ can be taken arbitrarily small, provided we make a suitably strong choice of separation constants in the specification of the tier regime, \eqref{eq:refHeightSep}.
\end{lemma}
\begin{proof}
For $1\leq i \leq l_j$, we wish to estimate \[\left|\frac{y_{L-L(j)+i}}{y_{L-L(j-1)}}\right|.\] As a consequence of the height estimation lemma, Lemma \ref{lem:heightEstProcedure}, we have that
\begin{equation}\label{eq:quotientBndForTierCoeffs}\left|\frac{y_{L-L(j)+i}}{y_{L-L(j-1)}}\right|\lesssim h(\mathcal{T}_j)^{D_{l_j-i}(\mathcal{T}_j)}.\end{equation}
We use this to bound the error term. 

We can write \[\widetilde{\Psi}(t)=\prod_{j=1}^s\frac{1}{y_{L-L(j-1)}}\left(\sum_{i_j=0}^{l_j}y_{L-L(j)+i_j}t^{k_{i_j}(\mathcal{T}_s)}\right).\]
To avoid a proliferation of nested sub and super-scripts, we will sometimes write $i(j)$ and $l(j)$ in place of $i_j$ and $l_j$, respectively. If we expand the above product expression for $\widetilde{\Psi}$, we obtain a sum that we will refer to throughout this proof. Each term in the resulting sum can be indexed by $(i_1,i_2,\ldots,i_s)$, where the index $i_j$ ranges over $\lbrace 0 ,1, \ldots, l_j\rbrace$ for each $1\leq j\leq s$. 

We first consider those terms which sum to $\Psi$. For $1\leq a\leq s$ one can see that the term indexed by $(0,0,\ldots,i_a,l_{a+1},l_{a+2},\ldots,l_{s})$, with $0\leq i_a<l_a$ is exactly \[\left(\prod_{j=1}^{a-1}\frac{y_{L-L(j)}}{y_{L-L(j-1)}}\right)\left(\frac{1}{y_{L-L(a-1)}}y_{L-L(a)+i(a)}t^{k_{i(a)}(\mathcal{T}_a)}\right)\left(\prod_{i=a+1}^{s}t^{k_{l(i)}(\mathcal{T}_i)}\right)\]
\[=y_{L-L(a)+i(a)}t^{k_{L-L(a)+i(a)}}.\]
As an example, corresponding to $a=s$, we have that the term indexed by $(0,0,\ldots,0,i_s)$ is
\[\left(\prod_{j=1}^{s-1}\frac{y_{L-L(j)}}{y_{L-L(j-1)}}\right)\left(\frac{1}{y_{L-L(s-1)}}y_{L-L(s)+i(s)}t^{k_{i(s)}(\mathcal{T}_s)}\right)\]
\[=y_{i(s)}t^{k_{i(s)}}.\]
We also have that the term indexed by $(l_1,l_2,\ldots,l_s)$ is 
\[\left(\prod_{i=1}^{s}t^{k_{l(i)}(\mathcal{T}_i)}\right)=t^{k_L}.\]
In this way, we have uniquely indexed all of the terms appearing in the expansion of $ \Psi(t)$. The remaining terms are exactly those which sum to $E(t)=\widetilde{\Psi}(t)- \Psi(t)$. Before proceeding to bound $E$, let us consider the size of the terms we have just indexed. This will inform us as to the bounds we should shoot for on the error term. For the term indexed by $(0,0,\ldots,i_a,l_{a+1},l_{a+2},\ldots,l_{s})$, we see using \eqref{eq:quotientBndForTierCoeffs} that it is bounded in magnitude
\[\lesssim\left(\prod_{j=1}^{a-1}h(\mathcal{T}_j)^{D_{l_j}(\mathcal{T}_j)}\right)\left(h(\mathcal{T}_a)^{D_{l(a)-i(a)}(\mathcal{T}_a)}|t|^{k_{i(a)}(\mathcal{T}_a)}\right)\left(\prod_{j=a+1}^{s}|t|^{k_{l(j)}(\mathcal{T}_j)}\right).\]
Now, if we are considering $h(\mathcal{T}_{r+1})\ll|t|\lesssim h(\mathcal{T}_r)$, for $1\leq a \leq L$, these terms can be uniformly bounded  
\[\lesssim\left(\prod_{j=1}^{r}h(\mathcal{T}_j)^{D_{l_j}(\mathcal{T}_j)}\right)\left(\prod_{j=r+1}^{s}|t|^{k_{l(j)}(\mathcal{T}_j)}\right)=\left(\prod_{j=1}^{r}h(\mathcal{T}_j)^{D(\mathcal{T}_j)}\right)|t|^{\sum_{j=r+1}^sD(\mathcal{T}_j)}.\]

In what follows, we consider those terms that we did not specify as summing to $\Psi$ above. We refer to these as the remainder terms. Let us fix $1\leq r \leq s$ and consider \begin{equation}\label{eq:tBetweenRandRplus1}h(\mathcal{T}_{r+1})\ll|t|\lesssim h(\mathcal{T}_r).\end{equation} 
As above, for the term indexed by $(i_1,i_2,\ldots,i_s)$, we can estimate the size of each of the $s$ factors using \eqref{eq:quotientBndForTierCoeffs} and \eqref{eq:tBetweenRandRplus1}. For $1\leq j \leq r$,
\begin{equation}\label{eq:boundFactorjleqr}\left|\frac{1}{y_{L-L(j-1)}}y_{L-L(j)+i(j)}t^{k_{i(j)}(\mathcal{T}_{j})}\right|\lesssim h(\mathcal{T}_{j})^{D(\mathcal{T}_{j})}.\end{equation}
For $r<j\leq s$,
\begin{equation}\label{eq:boundFactorjgtrr}\left|\frac{1}{y_{L-L(j-1)}}y_{L-L(j)+i(j)}t^{k_{i(j)}(\mathcal{T}_{j})}\right|\lesssim h(\mathcal{T}_{r})^{D(\mathcal{T}_{j})}.\end{equation}
To bound $E$ as an error term, we require stronger control on the remainder terms. We observe that those terms summing to $E$ are indexed by $(i_1,i_2,\ldots,i_s)$, such that if $a$ is the smallest index such that $i_{a}\neq 0$ (or $a=0$ if no such index $i_j$ exists), then there exists a minimal $a'>a$ for which $i_{a'}<l_{a'}$. The corresponding term is
\[\left(\prod_{j=1}^{a-1}\frac{y_{L-L(j)}}{y_{L-L(j-1)}}\right)\left(\frac{1}{y_{L-L(a-1)}}y_{L-L(a)+i(a)}t^{k_{i(a)}(\mathcal{T}_{a})}\right)\left(\prod_{j=a+1}^s\frac{1}{y_{L-L(j-1)}}y_{L-L(j)+i(j)}t^{k_{i(j)}(\mathcal{T}_j)}\right).\]
For each of these remainder terms, one of the factors appearing in the above expression will allow us to establish the error bound. Let us now fix some remainder term and its corresponding index $(i_1,\ldots,i_s)$. We split our analysis according to whether $r\leq a< s$ or $1\leq a< r$.

If $r\leq a<s$, then, using \eqref{eq:quotientBndForTierCoeffs} and the fact that $|t|\gg h(\mathcal{T}_{r+1})\geq h(\mathcal{T}_{a'})$, we can strongly bound the factor indexed by $i_{a'}$
\[\left|\frac{1}{y_{L-L(a'-1)}}y_{L-L(a')+i(a')}t^{k_{i(a')}(\mathcal{T}_{a'})}\right|\lesssim h(\mathcal{T}_{a'})^{D_{l(a')-i(a')}(\mathcal{T}_{a'})}
t^{k_{i(a')}(\mathcal{T}_{a'})}\]
\[\ll |t|^{k_{l(a')}}=|t|^{D(\mathcal{T}_{a'})},\]
since $D_{l(a')-i(a')}>0$. 
Putting this together with \eqref{eq:boundFactorjleqr} and \eqref{eq:boundFactorjgtrr}, we can bound the magnitude of the remainder term indexed by $(i_1,\ldots,i_s)$
\[\ll\left(\prod_{j=1}^{r}h(\mathcal{T}_j)^{D_{l_j}(\mathcal{T}_j)}\right)|t|^{D(\mathcal{T}_{a'})}\left|\prod_{j=r+1,j\neq a'}^{s}\frac{1}{y_{L-L(j-1)}}y_{L-L(j)+i(j)}t^{k_{i(j)}(\mathcal{T}_j)}\right|\]
\[\lesssim\left(\prod_{j=1}^{r}h(\mathcal{T}_j)^{D_{l_j}(\mathcal{T}_j)}\right)\left|\prod_{j=r+1}^{s}|t|^{D(\mathcal{T}_{j})}\right|.\]

If $1\leq a< r$, then we can strongly bound the factor indexed by $i_a$
\[\left|\frac{1}{y_{L-L(a-1)}}y_{L-L(a)+i(a)}t^{k_{i(a)}(\mathcal{T}_{a})}\right|\ll h(\mathcal{T}_a)^{k_{l(a)}}=h(\mathcal{T}_a)^{D(\mathcal{T}_a)},\]
since $|t|\lesssim h(\mathcal{T}_r)\ll h(\mathcal{T}_a)$ and $k_{i(a)}(\mathcal{T}_{a})> 0$. Therefore, also using \eqref{eq:boundFactorjleqr} and \eqref{eq:boundFactorjgtrr}, we can bound the remainder term indexed by $(i_1,\ldots,i_s)$
\[\ll\left(\prod_{j=1,j\neq a}^{r}h(\mathcal{T}_j)^{D_{l_j}(\mathcal{T}_j)}\right)\left(h(\mathcal{T}_a)^{D(\mathcal{T}_a)}\right)\left(\prod_{j=r+1}^s|t|^{D_{l_j}(\mathcal{T}_j)}\right).\]
Summing the bounds on each of the remainder terms gives the desired estimate.
\end{proof}

For a given tier, $\mathcal{T}_r$, there are two important scale parameters appearing in the previous analysis. Provided the separation in \eqref{eq:refHeightSep} is strong enough, we can set our fine parameter $\epsilon_f$ in Lemma \ref{lem:factErr} as small as we like. The other important parameter appearing in our analysis is the coarse scale parameter $\epsilon_c$, which we now define. We choose $\epsilon_c>0$ so that, according with Theorem \ref{thm:tierStrucGivenRegime}, \emph{at most} $L(\mathcal{T}_r)$ roots from $\mathcal{T}_r$ can appear in the ball $B(w,3\epsilon_c h(\mathcal{T}_r))$ for any choice of $w\in\mathcal{T}_r$. 

Before giving the proof of our main result, we require a covering lemma. This covering lemma essentially provides a partition of the roots into clusters of size up to $L$, with strong separation between distinct clusters. In place of clusters, which are finite collection of roots, we use cells, which are suitable open balls containing these roots. The proof gives a recursive construction of these cells. Associated with this construction are a well separated sequence of parameters, $\epsilon_f\ll \epsilon_1\ll \epsilon_2\ll \ldots\ll \epsilon_L\ll   \epsilon_{L+1}\ll\epsilon_c$, which we now define.

\begin{definition}
\label{def:SeqErrParameters}We set $\epsilon_1=\epsilon_f^{\frac{1}{L}}$ and, for $1\leq j \leq L$, we set $\epsilon_{j+1}=\epsilon_j^{\frac{1}{L}}$. 
\end{definition}
\begin{remark}
We can achieve strong separation of the parameters $\epsilon_j$ provided we start from a suitable fine error parameter, $\epsilon_f$. Indeed, to have that $\epsilon_{j}\ll \epsilon_{j+1}$ and $\epsilon_{L}\ll \epsilon_c$, the two things we require are that \[\epsilon_{j}/\epsilon_{j+1}=\epsilon_j^{1-\frac{1}{L}}=\epsilon_1^{\frac{L-1}{L^{j}}}\ll 1\]
and
\[\epsilon_{L+1}=\epsilon_1^{\frac{1}{L^L}}\ll\epsilon_c.\]
This is possible provided we can take $\epsilon_f$ sufficiently small, which is something we can achieve if we specify the tier regime with strong separation of the reference heights.
\end{remark}

\begin{definition}
For points $w_1,\ldots,w_a\in \CC$, we denote by $A(w_1,\ldots,w_a)$ their arithmetic mean:
\[A(w_1,\ldots,w_a)\coloneqq \frac{1}{a}\sum_{i=1}^{a}w_i.\]
\end{definition}
We can now state our root cell covering lemma.

\begin{lemma}\label{lem:cellCovering}
There exists a covering, $N(\mathcal{R})$, of the roots, $\mathcal{R}$, of $\Psi$ which satisfies the following.

Each connected component of $N(\mathcal{R})$, which we call a cell, $B$, contains only roots from one tier. Furthermore, if a cell $B$ contains only the roots $w_1,\ldots,w_b$ in the tier $\mathcal{T}_r$, then $B=B(A(w_1,\ldots,w_b),\epsilon_b h(\mathcal{T}_r))$. Each cell can contain at most $L$ roots. 

For distinct cells $B=B(A(w_1,\ldots,w_b),\epsilon_b h(\mathcal{T}_r))$ and $B'=B(A(w_1',\ldots,w_{b'}'),\epsilon_{b'} h(\mathcal{T}_r))$ with $b\geq b'$,
\[d(B,B')\geq \frac{1}{4}\epsilon_{b+1} h(\mathcal{T}_r).\]
In particular, for any root $w'$ outwith the cell $B=B(A(w_1,\ldots,w_b),\epsilon_b h(\mathcal{T}_r))$ and any root $w\in B$, 
\[|w-w'|\gtrsim \epsilon_{b+1} h(\mathcal{T}_r).\]
\end{lemma}

Figure \ref{fig:rootCellSketch} is a sketch of $4$ nearby root cells from a root cell covering. Roots are marked with a cross. Note that the larger cells are at a larger distance from adjacent cells.

\begin{figure}[h]
\includegraphics[width=0.9\textwidth]{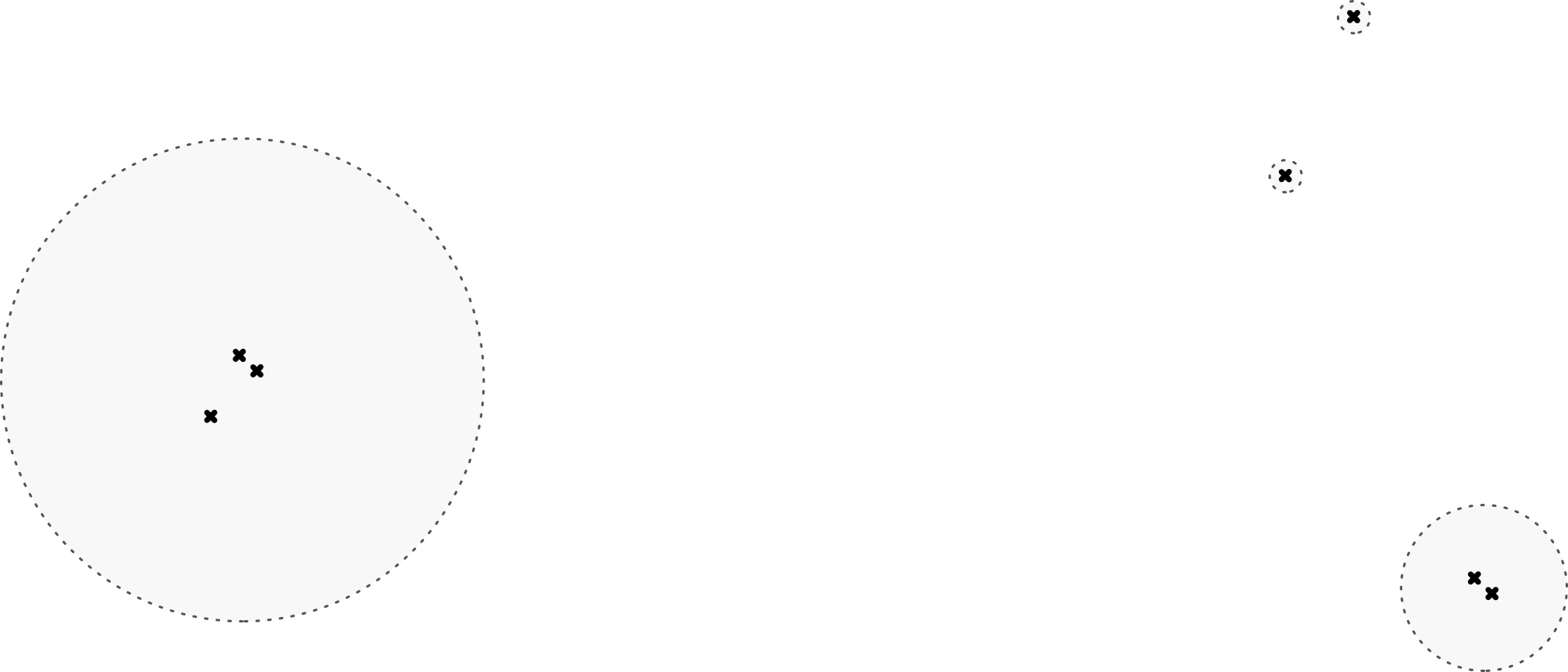}
\centering
\caption{Four cells from a root cell covering.}\label{fig:rootCellSketch}
\end{figure} 

\begin{proof}
Let us carry out the construction in each tier separately. Let $\mathcal{T}=\mathcal{T}_r$ for some $r$ and let $h=h(\mathcal{T}_r)$ denote the corresponding reference height. 

The construction is first outlined with reference to a particular choice of $w_1\in \mathcal{T}$. For each $w_1\in \mathcal{T}$, we construct an appropriate ball $N(w_1)$ containing $w_1$ and a number of other roots in $\mathcal{T}$. The root cell covering $N(\mathcal{R})$ in the lemma statement is then given as the union of the balls $N(w_1)$.

For now, let us fix some $w_1$. If for all remaining $w'\in\mathcal{T}$, $|w_1-w'|\geq  \epsilon_2 h$, then we set $N(w_1)=B(w_1,\epsilon_1 h)$.

Continuing the construction, it remains to consider the case where there exists  $w_2$ with $|w_1-w_2|< \epsilon_2 h$. Fix some choice of such $w_2$. We then divide our analysis with reference to $A\left( w_1,w_2\right)=\frac{w_1+w_2}{2}$. If, for all remaining $w'\in \mathcal{R}$, we have that $|w'-A\left( w_1,w_2\right)|\geq \epsilon_3h$, then we set $N(w_1)=B(A\left( w_1,w_2\right),\epsilon_2h)$. 
Otherwise, distinct from $w_1$ and $w_2$, there exists $w_3$ such that $|w_3-A\left( w_1,w_2\right)|\leq \epsilon_3h$ and we continue as previously, fixing some choice of $w_3$ and then working with reference to $A\left( w_1,w_2,w_3\right)$ and the scale $\epsilon_{4}h$.

The procedure continues; we consider the distance of remaining roots to the average of the roots already picked up by our procedure. Since there are finitely many roots, we know that the construction will terminate and we will refer to the resulting balls as terminal. For each choice of $w_1$, we construct a terminal ball $N(w_1)=B(A\left( w_1,w_2,\ldots,w_b\right), \epsilon_b h)$. By definition, the terminal ball $N(w_1)$ is constructed so that, for $w'\in\mathcal{R}$ with $w'\notin N(w_1)$, $\left|w'-A\left( w_1,w_2,\ldots,w_b\right)\right|\geq \epsilon_{b+1}h$. 

For each $w_1\in\mathcal{T}$, we can construct a ball $N(w_1)$ following the above procedure. The root cell covering $N(\mathcal{R})$ is simply given as the union of the sets $N(w_1)$. It remains to show that each connected component of $N(\mathcal{R})$ is given by a ball containing at most $L(\mathcal{T})$ roots, in particular that it is given by $N(w_1)$ for some $w_1$.

Let us first determine how close roots $w_j$ in a cell are to its centre. From the cell $B(A(w_1,\ldots,w_b),\epsilon_b h(\mathcal{T}_r))$, we take the root $w_a$. We see that \[|w_a-A(w_1,\ldots,w_b)|\leq|w_a-A(w_1,\ldots, w_{a}))|+\sum_{j=a+1}^{b}\left|A(w_1,\ldots, w_{j-1})-A(w_1,\ldots, w_{j})\right|\]
\[=|w_a-\frac{1}{a}\left((a-1)A(w_1,\ldots, w_{a-1})+w_a\right)|\]\[+\sum_{j=a+1}^{b}\frac{1}{j}\left|jA(w_1,\ldots, w_{j-1})-\left((j-1)A(w_1,\ldots, w_{j-1})+w_{j}\right)\right|\]
\[=\frac{a-1}{a}|w_a-A(w_1,\ldots, w_{a-1})|+\sum_{j=a+1}^{b}\frac{1}{j}\left|A(w_1,\ldots, w_{j-1})-w_{j}\right|\]
\[\leq\frac{a-1}{a}\epsilon_{a}h(\mathcal{T}_r)+\sum_{j=a+1}^{b}\frac{1}{j}\epsilon_jh(\mathcal{T}_r)\]
\begin{equation}\label{eq:dRootToCellCentre}\leq \frac{b-1}{b}\epsilon_bh(\mathcal{T}_r).\end{equation}

We can verify that the construction of the terminal balls takes at most $L(\mathcal{T})$ steps. Indeed, suppose that this were not the case and consider the step from $L(\mathcal{T})$ to $L(\mathcal{T})+1$. We then know we can find $L(\mathcal{T})+1$ roots $w_1\ldots w_{L(\mathcal{T})+1}$ contained in $B\left(A\left( w_1,w_2,\ldots,w_{L(\mathcal{T})+1}\right),\epsilon_{L(\mathcal{T})+1} h\right)$. This contradicts the regime specific structure theorem, Theorem \ref{thm:tierStrucGivenRegime}, since \[B\left(A\left( w_1,w_2,\ldots,w_{L(\mathcal{T})+1}\right),\epsilon_{L(\mathcal{T})+1} h\right)\subset B\left( w_1,3\epsilon_{L(\mathcal{T})+1} h\right) \subset B\left(w_1,\epsilon_c h\right)\] and $B\left(w_1,\epsilon_c h\right)$ contains at most $L(\mathcal{T})$ roots. 

We are now able to show that each cell of $N(\mathcal{R})$ is a ball. To this end, let us take two terminal balls $B=B(A(w_1,\ldots,w_b),\epsilon_b h(\mathcal{T}_r))$ and $B'=B(A(w_1',\ldots,w_{b'}'),\epsilon_{b'} h(\mathcal{T}_r))$ obtained by the construction outlined above. We suppose that these balls are such that \begin{equation}\label{eq:cellSetIntersect}B(A(w_1,\ldots,w_b),\epsilon_b h(\mathcal{T}_r))\cap B(A(w_1',\ldots,w_{b'}'),\epsilon_{b'} h(\mathcal{T}_r))\neq \emptyset.\end{equation}
Without loss of generality, suppose that $b'\leq b$. We find that, for any $w_{a'}\in \left\lbrace w_1',\ldots,w_{b'}'\right\rbrace$, \[|w_{a}'-A(w_1,\ldots,w_{b})|\]
\[\leq |w_{a}'-A(w_1',\ldots,w_{b'}')|+|A(w_1',\ldots,w_{b'}')-A(w_1,\ldots,w_{b})|\]
\[\leq \frac{b'-1}{b'}\epsilon_{b'}h(\mathcal{T}_r)+\epsilon_{b'}h(\mathcal{T}_r)+\epsilon_{b}h(\mathcal{T}_r)\]
\begin{equation}\label{eq:cellsIntersectOneSmaller}<\epsilon_{b+1}h(\mathcal{T}_r).\end{equation}
In particular, we must have that $w_{a}'\in \left\lbrace w_1,\ldots,w_{b}\right\rbrace$, because otherwise the construction we outlined above must continue to account for $w_{a'}$. Therefore $\left\lbrace w_1',\ldots,w_{b'}'\right\rbrace\subset \left\lbrace w_1,\ldots,w_{b}\right\rbrace$. It is then easy to verify that $B'\subset B$. This is obvious if $b=b'$. To show this when $b'<b$, let us take any complex number $w\in B'$. Relating $w$ to $B'$ and $B'$ to $w_{1}'$, which we know is an element of $B'$ and of $B$, we find that
\[d(w,A(w_1,\ldots,w_b))\]
\[\leq d(w,A(w_1',\ldots,w_{b'}'))+d(A(w_1',\ldots,w_{b'}'),w_1')+d(w_1',A(w_1,\ldots,w_b))\]
\[\leq \epsilon_{b'}h(\mathcal{T}_r)+\frac{b'-1}{b'}\epsilon_{b'}h(\mathcal{T}_r)+\frac{b-1}{b}\epsilon_{b}h(\mathcal{T}_r)\]
\[<\epsilon_{b}h(\mathcal{T}_r),\]
so that $w\in B$. Therefore, $B'\subset B$.

In fact, the above argument showing that cells are given by terminal balls can be strengthened. We can show that distinct cells are strongly separated. In particular, to conclude, we show that, for cells $B=B(A(w_1,\ldots,w_b),\epsilon_b h(\mathcal{T}_r))$ and $B'=B(A(w_1',\ldots,w_{b'}'),\epsilon_{b'} h(\mathcal{T}_r))$ with $b'\leq b$,
\[d(B,B')\gtrsim \epsilon_{b+1}h(\mathcal{T}_r).\]
Let us suppose, for a contradiction, that there exists a complex number
\begin{equation}\label{eq:fattenedCellSetIntersect}w\in B\left(A(w_1,\ldots,w_b),\frac{1}{3}\epsilon_{b+1} h(\mathcal{T}_r)\right)\cap B\left(A(w_1',\ldots,w_{b'}'),\frac{1}{3}\epsilon_{b'+1} h(\mathcal{T}_r)\right).\end{equation}
We find that
\[d(w_1',A(w_1,\ldots,w_b))\]
\[\leq d(w_1',A(w_1',\ldots,w_{b'}'))+d(A(w_1',\ldots,w_{b'}'),w)+d(w,A(w_1,\ldots,w_b))\]
\[\leq \frac{b'-1}{b'}\epsilon_{b'}h(\mathcal{T}_r)+\epsilon_{b'}h(\mathcal{T}_r)+\epsilon_{b}h(\mathcal{T}_r)\]
\[< \epsilon_{b+1}h(\mathcal{T}_r).\]
This inequality contradicts our assumption that $B=B(A(w_1,\ldots,w_b),\epsilon_b h(\mathcal{T}_r))$ was a terminal ball, since it implies that the construction should continue to account for the root $w_1'\notin \left\lbrace w_1,\ldots,w_{b}\right\rbrace$. Therefore \eqref{eq:fattenedCellSetIntersect} can not hold and, in particular, 
\[d(B,B')\geq \frac{1}{4}\epsilon_{b+1}h(\mathcal{T}_r).\]
\end{proof}

\begin{lemma}\label{lem:factorRootSize}
The roots, $\widetilde{\mathcal{T}_j}$, of the polynomial $\widetilde{\Psi}_j(t)$ are all comparable in magnitude to $h(\mathcal{T}_j)$.
\end{lemma}
\begin{proof}
This is a simple consequence of Lemma \ref{lem:heightEstProcedure}, which can be applied with reference to $\Psi$ or to $\widetilde{\Psi}_j$. In either case, we see exactly the same expressions appearing as height estimates and the lemma tells us that these are comparable to $h(\mathcal{T}_j)$.
\end{proof}

We are now ready to prove Theorem \ref{thm:approxFact}.
\begin{proof}
In the case that $s=1$, the factorisation is our initial polynomial and there is nothing to prove. Henceforth, we suppose that $s>1$.

We make use of Lemma \ref{lem:factorRootSize} and consider the factorisations of $ \Psi(t)$ and $\widetilde{\Psi}(t)$ in terms of their roots:
\begin{equation}\label{eq:PsiandTildePsiRootFactorisation} \Psi(t)=\prod_{j=1}^s\prod_{w\in\mathcal{T}_j}(t-w)\quad\text{and}\quad\widetilde{\Psi}(t)=\prod_{j=1}^{s}\widetilde{\Psi}_j(t)=\prod_{j=1}^{s}\prod_{\tilde{w}\in\widetilde{\mathcal{T}}_j}(t-\tilde{w}).\end{equation}

 Throughout, we appeal to the regime specific structure theorem, Theorem \ref{thm:tierStrucGivenRegime}, and the root cell covering lemma, Lemma \ref{lem:cellCovering}. Let us consider a specific $t\notin N(\mathcal{R})$ with $|t|\sim h(\mathcal{T}_r)$. By the definition of our covering in Lemma \ref{lem:cellCovering}, the closest $t$ can be to a root $w\in \mathcal{T}_r$ is $\epsilon_1h(\mathcal{T}_r)$ and, furthermore, we can check that there are at most $L(\mathcal{T}_r)$ roots $w$ with $\epsilon_1h(\mathcal{T}_r)<|t-w|<\epsilon_ch(\mathcal{T}_r)$. Indeed, if $w\in\mathcal{T}_r\cap B(t,\epsilon_c h( \mathcal{T}_r))$, then $B(t,\epsilon_c h( \mathcal{T}_r))\subset B(w,3\epsilon_c h( \mathcal{T}_r))$, which can contain at most $L(\mathcal{T}_r)$ roots, by Theorem \ref{thm:tierStrucGivenRegime} and our definition of $\epsilon_c$. The remaining roots in $\mathcal{T}_r$ are roots $w'\notin B(t,\epsilon_c h( \mathcal{T}_r))$ and we know there are at least $D(\mathcal{T}_r)-L(\mathcal{T}_r)$ of these. We thus find, using the factorisation \eqref{eq:PsiandTildePsiRootFactorisation}, that, for $t\notin N(\mathcal{R})$ with $|t|\sim h(\mathcal{T}_r)$, \[| \Psi(t)|\gtrsim \left(\epsilon_c^{D(\mathcal{T}_{r})-L(\mathcal{T}_r)}\epsilon_1^{L(\mathcal{T}_r)} h(\mathcal{T}_r)^{D(\mathcal{T}_{r})}\right)\prod_{i=r+1}^{s}h(\mathcal{T}_r)^{D(\mathcal{T}_i)}\prod_{i=1}^{r-1}h(\mathcal{T}_i)^{D(\mathcal{T}_i)}\] 
 \[\gg\epsilon_f h(\mathcal{T}_r)^{D(\mathcal{T}_{r})+D(\mathcal{T}_{r+1})+\ldots+D(\mathcal{T}_{s})}\prod_{i=1}^{r-1}h(\mathcal{T}_i)^{D(\mathcal{T}_i)},\]
provided $\epsilon_f$ has been taken small enough, since $\epsilon_1^{L(\mathcal{T}_r)}\epsilon_c^{D(\mathcal{T}_r)-L(\mathcal{T}_r)}= \epsilon_{f}^{\frac{L(\mathcal{T}_r)}{L}}\epsilon_c^{D(\mathcal{T}_r)-L(\mathcal{T}_r)}\gg \epsilon_f$. We now have an explicit lower estimate on the size of $ \Psi(t)$ for $t\notin N(\mathcal{R})$. It is now possible to make sense of Lemma \ref{lem:factErr} as an error expression. Indeed, for $|t|\sim h(\mathcal{T}_r)$, we see that
\begin{equation}\label{eq:errEPsiTildePsi}|E(t)|\leq \epsilon_f  h(\mathcal{T}_r)^{D(\mathcal{T}_{r})+D(\mathcal{T}_{r+1})+\ldots+D(\mathcal{T}_{s})}\prod_{i=1}^{r-1}h(\mathcal{T}_i)^{D(\mathcal{T}_i)},\end{equation} 
so that for $|t|\sim h(\mathcal{T}_r)$ with $t\notin N(\mathcal{R})$, 
\[|E(t)|\ll | \Psi(t)|.\] 
In particular, for $t\notin N(\mathcal{R})$, \[\widetilde{\Psi}(t)=
 \Psi(t)+E(t)\neq 0\]
so that, for the roots of $\widetilde{\Psi}$,
\[\widetilde{\mathcal{R}}\subset N(\mathcal{R}).\]

It remains for us to show that each cell of $N(\mathcal{R})$ contains the same number of roots of $\Psi$ and  $\widetilde{\Psi}$. If there was no error term and all roots were isolated, this would be easy as both functions would be equal to zero on $\mathcal{R}\subset N(\mathcal{R})$. In fact, the argument requires more precision. In order to account for the error term and cells containing multiple roots, we must consider the size of the functions $ \Psi$ and $\widetilde{\Psi}$ at a suitable distance from the roots $\mathcal{R}$. In particular, we estimate the size of the functions close to the boundary of $N(\mathcal{R})$. We know that $\mathcal{R}\subset N(\mathcal{R})$ and also that $\widetilde{\mathcal{R}}\subset N(\mathcal{R})$, which will allow us to estimate the size of the functions using their factorisations, \eqref{eq:PsiandTildePsiRootFactorisation}.

For the remainder of the proof, we fix a cell $B=B(u,R)$, containing $m$ roots, with centre $u$ and radius $R=\epsilon_mh(\mathcal{T}_r)$. We consider the size of the functions $ \Psi$ and $\widetilde{\Psi}$ at the boundary of $B^*$, where $B^*=B(u,2R)$ is the double of $B$. For $t\in \partial B^*$, using the factorisation \eqref{eq:PsiandTildePsiRootFactorisation} and the fact that $d(t, B)=\epsilon_mh(\mathcal{T}_r)$ with $B$ containing $m$ roots, 
\begin{equation}\label{eq:PsitUpper}| \Psi(t)|\lesssim\epsilon_m^m h(\mathcal{T}_r)^{D(\mathcal{T}_s)+\ldots+D(\mathcal{T}_r)}\prod_{i=1}^{r-1}h(\mathcal{T}_i)^{D(\mathcal{T}_i)}.\end{equation}

By the covering Lemma \ref{lem:cellCovering}, we can verify that points on $\partial B^*$ are well separated from cells other than $B$. Indeed, $d(\partial B^*, B)=R=\epsilon_mh(\mathcal{T}_r)$ and, if we take a distinct cell $B'$, then Lemma \ref{lem:cellCovering} tells us that $d(B,B')\geq\frac{1}{4}\epsilon_{m+1}h(\mathcal{T}_r)$ so that $d(\partial B^*, B')\gtrsim \epsilon_{m+1}h(\mathcal{T}_r)$. In particular, since $\widetilde{\mathcal{R}}\subset N(\mathcal{R})$, for $t\in \partial B^*$ and any root $\tilde{w}'\in \widetilde{\mathcal{T}}_r$ with $\tilde{w}'\notin B$,
\begin{equation}\label{eq:distRootsInOtherCells}|w'-t|\gtrsim \epsilon_{m+1}h(\mathcal{T}_r).\end{equation}
 We also note that, by Theorem \ref{thm:tierStrucGivenRegime}, for $t\in \partial B^*$, $B(t,\epsilon_ch(\mathcal{T}_r))$ can contain at most $L(\mathcal{T}_r)$ roots of $\widetilde{\Psi}$. Let us now suppose that the given cell, $B$, contains only $\tilde{m}<m$ roots of $\widetilde{\Psi}$. For $t\in \partial B^*\subset N(\mathcal{R})^c$, there can be at most $L(\mathcal{T}_r)-\tilde{m}$ roots $\tilde{w}'\notin B$ for which $\tilde{w}' \in B(t,\epsilon_c h(\mathcal{T}_r))$.  We know that, for $t\in \partial B^*$, $|t|\sim h(\mathcal{T}_r)$. Therefore, for $t\in \partial B^*$, using the factorisation \eqref{eq:PsiandTildePsiRootFactorisation} and also the distance estimate \eqref{eq:distRootsInOtherCells},
 \[\left|\widetilde{\Psi}(t)\right|\gtrsim \left(\epsilon_c^{D(\mathcal{T}_r)-L(\mathcal{T}_r)}\epsilon_m^{\tilde{m}}\epsilon_{m+1}^{L(\mathcal{T}_r)-\tilde{m}} h(\mathcal{T}_r)^{D(\mathcal{T}_r)}\right)h(\mathcal{T}_r)^{D(\mathcal{T}_{s})+\ldots+D(\mathcal{T}_{r+1})}\prod_{i=1}^{r-1}h(\mathcal{T}_i)^{D(\mathcal{T}_i)}\]
  \[\gtrsim \left(\epsilon_c^{D(\mathcal{T}_r)-L(\mathcal{T}_r)}\epsilon_m^{m-1}\epsilon_{m+1}^{L(\mathcal{T}_r)-(m-1)} h(\mathcal{T}_r)^{D(\mathcal{T}_r)}\right)h(\mathcal{T}_r)^{D(\mathcal{T}_{s})+\ldots+D(\mathcal{T}_{r+1})}\prod_{i=1}^{r-1}h(\mathcal{T}_i)^{D(\mathcal{T}_i)}\]
\begin{equation}\label{eq:TildePsiLower}\gg\epsilon_m^{m} h(\mathcal{T}_r)^{D(\mathcal{T}_s)+\ldots+D(\mathcal{T}_r)}\prod_{i=1}^{r-1}h(\mathcal{T}_i)^{D(\mathcal{T}_i)},\end{equation}
because $\epsilon_c^{D(\mathcal{T}_r)-L(\mathcal{T}_r)}\epsilon_{m+1}^{L(\mathcal{T}_r)-(m-1)}\gg \epsilon_m$. This contradicts \eqref{eq:errEPsiTildePsi}: considering \eqref{eq:PsitUpper} and \eqref{eq:TildePsiLower} together, we have that \[|E(t)|\gtrsim \left|\widetilde{\Psi}(t)\right|-\left|\Psi(t)\right|\gg \epsilon_m^{m} h(\mathcal{T}_r)^{D(\mathcal{T}_s)+\ldots+D(\mathcal{T}_r)}\prod_{i=1}^{r-1}h(\mathcal{T}_i)^{D(\mathcal{T}_i)}.\] Therefore, $B$ must contain $m$ roots of $\widetilde{\Psi}$. 
\end{proof}

\part{Oscillatory integral estimates}\label{part:oscInt}
In this section, we prove the oscillatory integral estimates given by Theorems \ref{thm:oscIntEst} and \ref{thm:oscIntEstFine}. Let us first present an example which shows why the condition $k_{L-s}\geq s+1$ is required in Theorem \ref{thm:oscIntEstFine}. 

\begin{proposition}\label{prop:oscIntSharpEx}For $k\geq m$, define $\Phi(t)$ by \begin{equation}\label{eq:unifEstExt}\Phi(0)=0\text{ and }\Phi'(t)=y_L(t^{k}-1)^{L-m+1}.\end{equation}
We set $y_1=\ldots y_{m-1}=0$, the remaining $(x,y)$ parameters are defined implictly by $\Phi'(t)=x+\sum_{j=m}^Ly_jt^{(j-m+1)k}$.
For this polynomial, we have that \[\left|\int e^{i\Phi(t)}dt\right|\gtrsim |y_m|^{-\frac{1}{L-m+2}},\]
for $(x,y)$ in a region $R$ containing arbitrarily large $y_m$. In particular, for the estimate \eqref{eq:oscIntEstI} to hold in the region $R$, we require that $k=k_m\geq L-m+1$.
\end{proposition}
\begin{remark}Proposition \ref{prop:oscIntSharpEx} is a direct consequence of the following (equivalent) proposition, which amounts to a change in notation, and our testing the inequality
\[|y_m|^{-\frac{1}{L-m+2}}\lesssim |y_m|^{-\frac{1}{k_m+1}},\]
 as $|y_m|\rightarrow \infty$. The inequality leads to the necessary condition $k_m\geq L-m+1$.
\end{remark}
\begin{proposition*}Here, we set $\tilde{L}=L-m+1$. For $k\geq m$, define $\Phi(t)$ by \begin{equation}\label{eq:unifEstExt2}\Phi(0)=0\text{ and }\Phi'(t)=y_{\tilde{L}}(t^{k}-1)^{l}.\end{equation}
The $(x,y)$ parameters are defined implictly by $\Phi'(t)=x+\sum_{j=1}^{\tilde{L}}y_jt^{jk}$.
For this polynomial, we have that \[\left|\int e^{i\Phi(t)}dt\right|\gtrsim |y_1|^{-\frac{1}{l+1}},\]
for $(x,y)$ in a region $R$ containing arbitrarily large $y_1$. 
\end{proposition*}
We refer the reader to \cite{mythesis} for a proof, in which we work via a standard non-stationary phase analysis.

Let us now turn to the proof of the oscillatory integral bounds of Hickman and Wright.
\begin{proof}[Proof of Theorems \ref{thm:oscIntEst} and \ref{thm:oscIntEstFine}]
We first give the proof in the case that $k_m\geq L$ (Theorem \ref{thm:oscIntEst}) and later give the technical case splitting required to obtain the result when $k_m\geq L-m+1$ (Theorem \ref{thm:oscIntEstFine}). 

Without loss of generality, we can prove the result for $\Phi$ as in \eqref{eq:polyPhaseDef} chosen such that $\max_{1\leq j\leq L}{|y_j|}=1$. To see this, suppose we have a phase $\widetilde{\Phi}=\widetilde{\Phi}_{\tilde{x},\tilde{y}}$ of the same form as \eqref{eq:polyPhaseDef},  where $\tilde{y}\in\RR^L$ is unrestricted. We wish to show $|J(\tilde{x},\tilde{y})|\lesssim \min_{j}|\tilde{y}_j|^{-\frac{1}{k_j+1}}$,
where $J(\tilde{x},\tilde{y})=\int e^{i\tilde{\Phi}(s)}ds$, with $\tilde{\Phi}(s)=\tilde{x}s+\frac{\tilde{y}_1}{k_1+1}s^{k_1+1}+\ldots+\frac{\tilde{y}_L}{k_L+1}s^{k_L+1}$.
By making a change of variables in the $s$ coordinate, it suffices to prove $|J(x,y)|\lesssim 1$,
where $J(x,y)=\int e^{i\Phi(s)}ds$, with $\Phi(s)=xs+\frac{y_1}{k_1+1}s^{k_1+1}+\ldots+\frac{y_L}{k_L+1}s^{k_L+1}$ such that $\max_{1\leq j \leq L}|y_j|=1$. Indeed, we set $\sigma =\max_{j}|\tilde{y}_j|^{\frac{1}{k_j+1}}$ and make the change of variables $\sigma s= t$ in the integral expression for $J$. Writing $\widetilde{\Phi}(s)$ in terms of $t$ we see that \[\widetilde{\Phi}(t)=\tilde{x}t+\frac{\tilde{y}_1}{k_1+1}t^{k_1+1}+\ldots+\frac{\tilde{y}_L}{k_L+1}t^{k_L+1}\]
\[=xs+\frac{y_1}{k_1+1}s^{k_1+1}+\ldots+\frac{y_L}{k_L+1}s^{k_L+1},\]
where $x=\sigma^{-1}\tilde{x}$ and $y_j=\sigma^{-(k_j+1)}\tilde{y}_j$. By definition of $\sigma$, $\max_{1\leq j \leq L}|y_j|=1$.

We now use the root structure Theorem \ref{thm:tierStruc} to prove the desired inequality \[\left|J(x,y)\right|=\left|\int e^{i\Phi(t)}dt\right|\lesssim 1.\]

Let us consider the case where $\max_{1\leq j \leq L}|y_j|^{\frac{1}{k_j}}=|y_m|^{\frac{1}{k_m}}=1$. Here, the appropriate cluster estimate to consider is the $k_m$-cluster estimates, namely the bound we desire is obtained with a suitable choice of $\mathcal{C}$ such that $|\mathcal{C}|=k_m$. We work to show that, for any $z_j$, there exists a $k_m$-cluster $\mathcal{C}$ such that \begin{equation}\label{eq:clusterCompProdRoots}\prod_{l\notin \mathcal{C}}|z_j-z_l|\gtrsim |z_1z_2\ldots z_{k_L-k_m}|.\end{equation}Once this has been established we note that \[ |z_1z_2\ldots z_{k_L-k_m}|\gtrsim |S_{k_L-k_m}|=\left| \frac{y_m}{y_L}\right|=\left| \frac{1}{y_L}\right|\] so that we can apply the Phong and Stein estimate, Theorem \ref{thm:stePho}, to establish that \[\left|J(x,y)\right|\lesssim \left(\frac{1}{|y_Lz_1z_2\ldots z_{k_L-k_m}|}\right)^\frac{1}{k_m+1}\lesssim 1.\]

By Theorem \ref{thm:tierStruc}, given a root of $\Phi'$, $z$, there are \emph{at most} $L$ other roots in $B(z,\epsilon|z|)$, where $\epsilon$ is some suitable small constant. For a root $z_j\notin B(z,\epsilon|z|)$, \begin{equation}\label{eq:sizeDiffRoots}|z_j-z|\geq \max\lbrace \epsilon |z|, |z_j|-|z|\rbrace.\end{equation}

We now take an arbitrary root $z$ and construct an appropriate $k_m$-cluster containing $z$. The cluster we will construct will depend on what tier $z$ is in. If there are not enough roots smaller than $z$, then we will just put the smallest $k_m$ roots in the cluster. As we will see, this will necessarily include all those roots in $B(z,\epsilon|z|)$. If there are many roots smaller than $z$, we will choose our cluster to contain all roots in $B(z,\epsilon|z|)$, with the remaining elements taken to be any small roots.

Recall that $D(\mathcal{T}_i)=|\mathcal{T}_i|$. Let $r$ be chosen such that $z\in\mathcal{T}_r$. In the case that $D(\mathcal{T}_{r})+D(\mathcal{T}_{r+1})+\ldots+D(\mathcal{T}_{s})\leq k_m$ then we choose our cluster $\mathcal{C}=\left\lbrace z_{k_L},z_{k_L-1},\ldots, z_{k_L-k_m+1}\right \rbrace\supset \mathcal{T}_{r}\cup\mathcal{T}_{r+1}\cup\ldots\cup\mathcal{T}_{s}$ so that 
\[\prod_{j\notin \mathcal{C}}|z-z_j|\sim |z_{1}z_{2}\ldots z_{k_L-k_m}|.\]

It remains to consider the case that $D(\mathcal{T}_{r})+D(\mathcal{T}_{r+1})+\ldots+D(\mathcal{T}_{s}) > k_m$. After taking roots $\mathcal{C}_r=B(z,\epsilon|z|)\cap \mathcal{T}_r$, any choice of $k_m-|\mathcal{C}_r|$ roots from $ \mathcal{T}_r\cup \ldots \cup  \mathcal{T}_s$ suffices to complete our cluster $\mathcal{C}$. Indeed, we find, by \eqref{eq:sizeDiffRoots}, that
\[\prod_{z_j\notin\mathcal{C}}|z-z_j|\]
\[=\left(\prod_{z_j\in \mathcal{T}_{1}\cup\mathcal{T}_{2}\cup\ldots\cup\mathcal{T}_{r-1}}|z-z_j|\right)\left(\prod_{j\notin \mathcal{C},z_j\in \mathcal{T}_r\cup\mathcal{T}_{r+1}\cup\ldots \cup\mathcal{T}_{s}}|z-z_j|\right)\]
\[\gtrsim_{\epsilon} \left(|z_1z_2\ldots z_{D(\mathcal{T}_1)+\ldots +D(\mathcal{T}_{r-1})}|\right)\left(|z|^{D(\mathcal{T}_{r})+\ldots +D(\mathcal{T}_{s})-k_m}\right)\]
\[\gtrsim |z_1z_2\ldots z_{k_L-k_m}|.\]

For these calculations to be valid we require that there are enough spaces in $\mathcal{C}$ to contain all of those roots in $B(z,\epsilon|z|)$, we require $|\mathcal{C}|=k_m\geq L$, which is true by assumption. 

Now, we consider the case where we can weaken the condition on $k_m$ to $k_m\geq L-m+1$. It is here we apply Theorem \ref{thm:tierStrucFine}. Recall $|y_m|=1$. Set $\delta_0=1$. Since $|y_m|\geq\delta_0$, Theorem \ref{thm:tierStrucFine} applies relative to $|y_m|$ and $\delta_0$ provided that, \[\delta_1 \geq |y_n|^\frac{1}{k_n},\quad \text{for}\quad n>m,\] with a suitable constant $\delta_1=\delta(\delta_0)>0$. However, in general we only have that  $1 \geq |y_n|^\frac{1}{k_n}$ for $n>m$. Nevertheless, we will still be able to apply Theorem \ref{thm:tierStrucFine} by a suitable inductive procedure. Let $m(0)=m$, $\delta_0=1$, and $\delta_1=\delta(\delta_0)$ be as above. Suppose, for induction, that $m(0)<m(1)<\ldots<m(j)$ and $\delta_0,\delta_1,\ldots,\delta_{j}$ have already been defined by the inductive procedure and that $|y_{m(j)}|^{\frac{1}{k_{m(j)}}}>\delta_{j}=\delta(\delta_{j-1})$. There are two cases. In the first case, the coefficients are such that \[|y_n|^{\frac{1}{k_n}}\leq \delta_{j+1},\quad \text{for}\quad n>m(j),\] where $\delta_{j+1}=\delta(\delta_j)$ is such that Theorem \ref{thm:tierStrucFine} applies relative to the coefficient $|y_{m(j)}|>\delta_{j}^{k_{m(j)}}$. If we are in this case, we terminate the inductive procedure. Otherwise, in the second case, there exists some $m(j+1)>m(j)$ such that $|y_{m(j+1)}|^{\frac{1}{k_{m(j+1)}}}>\delta_{j+1}$, and we proceed with the induction. The process must terminate, since there are finitely many coefficients. We can thus apply Theorem \ref{thm:tierStrucFine} relative to some $|y_{m'}|^{\frac{1}{k_{m'}}}>\gamma'$, for some $m'\geq m$, with \[|y_n|^{\frac{1}{k_n}}\leq \delta(\gamma'),\quad \text{for}\quad n>m'.\]

By Theorem \ref{thm:tierStrucFine} we know that at most $L-m'+1\leq L-m+1$ roots can be contained in $B(w, \epsilon|w|)$ for roots $w$ in the large tiers $\mathcal{T}_r$. It is also a consequence of that theorem that there are at most $k_{m'}$ roots in the small tiers, with such roots $|w|\lesssim_{\delta_{m'}}1$. By our assumption, we also know that $k_{m}\geq L-m+1$ so that $k_{m'}\geq L-m'+1$. The above argument thus carries through, constructing appropriate $k_{m'}$-clusters $\mathcal{C}$ such that  \[\prod_{l\notin \mathcal{C}}|z_j-z_l|\gtrsim |z_1z_2\ldots z_{k_L-k_{m'}}|.\]
Once this has been established we note that \[ |z_1z_2\ldots z_{k_L-k_{m'}}|\gtrsim |S_{k_L-k_{m'}}|=\left| \frac{y_{m'}}{y_L}\right|\gtrsim_{\delta_1,\ldots,\delta_{m'-1}}\left| \frac{1}{y_L}\right|.\]
The result follows by applying the Phong-Stein estimate, Theorem \ref{thm:stePho}.
\end{proof}

If we further restrict the region we consider, it is possible to strengthen the oscillatory integral bound \eqref{eq:oscIntEstI}.
\begin{proposition} 
Set \[I(x,y)=\int e^{i\Phi(t)}dt.\]
Then, it is possible to bound \[|I(x,y)|\lesssim \epsilon \min|y_j|^{-\frac{1}{k_j+1}},\]
provided we take $|x|\gg \max|y_j|^{\frac{1}{k_j+1}}$ with a sufficiently large constant depending on $\epsilon$.
\end{proposition}
\begin{proof}
We rescale as we have done previously, setting $s=\sigma t$ where $\sigma=\max_j|y_j|^\frac{1}{k_j+1}$ and $\tilde{y_j}=\sigma^{-(k_j+1)}y_j$. Note that, after rescaling, $|\tilde{x}|\gg 1$ by our assumption. It suffices for us to prove that \[\left|J(\tilde{x},\tilde{y})\right|=\left|\int e^{i\widetilde{\Phi}(s)}ds\right|\lesssim \epsilon,\]
provided $|\tilde{x}|\gg 1$ with a large enough constant.

By Lemma \ref{lem:heightEstProcedure}, we know that all roots of $\widetilde{\Phi}'$ are comparable to \[\eta=\left|\frac{\tilde{x}}{\tilde{y}_L}\right|^\frac{1}{k_L}.\]

Given a root $z$, we can construct a singleton cluster such that \[\prod_{j\notin \mathcal{C}}|z-z_j|\sim \eta^{k_L-1}\gg \eta^{k_L-k_m}\gtrsim|S_{k_L-k_m}|=\left|\frac{1}{\tilde{y}_L}\right|.\]
Applying the Phong-Stein bound then gives the required result.
\end{proof}

\bibliographystyle{plain}
\bibliography{rtStrREFS}
\end{document}